\documentclass[12pt,a4paper]{amsart}

\usepackage[a4paper, total={6in, 23.7cm}]{geometry}
\usepackage{amsmath,amsthm,amsfonts,amssymb,mathrsfs,amsbsy}
\usepackage{graphicx}
\usepackage{color}
\usepackage[noadjust]{cite}
\usepackage{marginnote}
\usepackage{xifthen}
\usepackage{mathtools}
\usepackage{dsfont}
\usepackage{bbm}
\usepackage{appendix}
\usepackage{delarray}
\usepackage{enumerate}
\usepackage{enumitem}

\usepackage{xcolor}
\usepackage{tikz}
\usepackage[hypertexnames=false]{hyperref} 
\usepackage{autonum}

\binoppenalty=\maxdimen
\relpenalty=\maxdimen

\newenvironment{spm}
{\bigl(\begin{smallmatrix}}
	{\end{smallmatrix}\bigr)}

\newenvironment{nalign}
{\begin{equation}\begin{aligned}}
		{\end{aligned}\end{equation}\ignorespacesafterend}

\newtheorem{theorem}{Theorem}

\newtheorem{defi}[theorem]{Definition}

\newtheorem{lemma}[theorem]{Lemma}
\newtheorem{prop}[theorem]{Proposition}

\numberwithin{equation}{section}
\numberwithin{theorem}{section}
\numberwithin{figure}{section}

\theoremstyle{remark}
\newtheorem{rem}[theorem]{Remark}

\theoremstyle{remark}

\theoremstyle{theorem}
\newtheorem{assumption}[theorem]{Assumption}

\newcommand{\rev}[1]{{ #1}}
\newcommand{\revone}[1]{{\color{black} #1}}

\newcommand\R{\mathbb{R}}
\newcommand\C{\mathbb{C}}

\newcommand\bu{\pmb u}
\newcommand\bff{\pmb f}
\newcommand\bA{\pmb A}
\newcommand\bB{\pmb B}
\newcommand\bC{\pmb C}

\newcommand\dx{\,{\rm d}x}
\newcommand\ds{\,{\rm d}s}
\newcommand\dtau{\,{\rm d}\tau}

\newcommand{\UV}{(\pmb U, V)}
\newcommand{\UVx}{\big(\pmb U(x), V(x)\big)}

\newcommand\V{\mathcal{V}}

\newcommand{\oo}{\overline{\Omega}}

\renewcommand\L{\mathcal{L}}
\renewcommand\Re{{\rm Re \,}}

\DeclareMathOperator*{\esssup}{ess\,sup}
\newcommand{\retainlabel}[1]{\label{#1}\sbox0{\ref{#1}}}

\newcommand{\W}[2]
{{
		W^{#1,#2}_
		{
			\ifthenelse
			{
				\equal{#1}{2}
			}
			{
				\nu
			}
			{
			}
		}(\Omega)
}}

\newcommand{\Li}{{L^\infty(\Omega)}}
\newcommand{\LiR}{{L^\infty(\R)}}
\newcommand{\Lp}[1]
{{
		L^
		{
			\ifthenelse
			{
				\isempty{#1}
			}
			{
				p
			}
			{
				#1
			}
		}(\Omega)
}}

\definecolor{lime}{HTML}{A6CE39}
\DeclareRobustCommand{\orcidicon}{%
	\begin{tikzpicture}
		\draw[lime, fill=lime] (0,0) 
		circle [radius=0.16] 
		node[white] {{\fontfamily{qag}\selectfont \tiny ID}};
		\draw[white, fill=white] (-0.0625,0.095) 
		circle [radius=0.007];
	\end{tikzpicture}
	\hspace{-2mm}
}

\definecolor{lime}{HTML}{A6CE39}
\DeclareRobustCommand{\orcidicon}{%
	\begin{tikzpicture}
		\draw[lime, fill=lime] (0,0) 
		circle [radius=0.16] 
		node[white] {{\fontfamily{qag}\selectfont \tiny ID}};
		\draw[white, fill=white] (-0.0625,0.095) 
		circle [radius=0.007];
	\end{tikzpicture}
	\hspace{-2mm}
}

\begin{document} 

\title[Reaction-diffusion-ODE systems]{
Stable discontinuous stationary solutions \\ to reaction-diffusion-ODE systems}

\author[S. Cygan]{Szymon Cygan \href{https://orcid.org/0000-0002-8601-829X}{\orcidicon}}
\address[S. Cygan]{	Instytut Matematyczny, Uniwersytet Wroc\l{}awski, pl. Grunwaldzki 2/4, \hbox{50-384} Wroc\l{}aw, Poland \\ 
\href{https://orcid.org/0000-0002-8601-829X}{orcid.org/0000-0002-8601-829X}}
\email{szymon.cygan@math.uni.wroc.pl}
\urladdr {http://www.math.uni.wroc.pl/~scygan}

\author[A. Marciniak-Czochra]{Anna Marciniak-Czochra
\href{https://orcid.org/0000-0002-5831-6505}{\orcidicon}}
\address[A. Marciniak-Czochra]{
Institute of Applied Mathematics,  Interdisciplinary Center for Scientific Computing (IWR), Heidelberg University, 69120 Heidelberg, Germany\\
\href{https://orcid.org/0000-0002-5831-6505}{orcid.org/0000-0002-5831-6505}}
\email{anna.marciniak@iwr.uni-heidelberg.de}
\urladdr {http://www.biostruct.uni-hd.de/}

\author[G. Karch]{Grzegorz Karch \href{https://orcid.org/0000-0001-9390-5578}{\orcidicon}}
\address[G. Karch]{	
Instytut Matematyczny, Uniwersytet Wroc\l{}awski, pl. Grunwaldzki 2/4, \hbox{50-384} Wroc\l{}aw, Poland \\ 
\href{https://orcid.org/0000-0001-9390-5578}{orcid.org/0000-0001-9390-5578}}
\email{grzegorz.karch@math.uni.wroc.pl}
\urladdr {http://www.math.uni.wroc.pl/~karch}

\author[K. Suzuki]{Kanako Suzuki\href{https://orcid.org/0000-0001-8018-9621}{\orcidicon}}
\address[K. Suzuki]{
Graduate School of Science and Engineering, Ibaraki University,
2-1-1 Bunkyo, Mito 310-8512, Japan\\
\href{https://orcid.org/0000-0001-8018-9621}{orcid.org/0000-0001-8018-9621}
}
\email{kanako.suzuki.sci2@vc.ibaraki.ac.jp}

\date{\today}

\thanks{The authors appreciated helpful comments on this work by Dr.~Chris Kowall.
	S.~Cygan acknowledges a  support  by the Polish NCN grant 2016/23/B/ST1/00434.  The work of A.
	Marciniak-Czochra was supported by the Deutsche Forschungsgemeinschaft (DFG, German Research Foundation) under Germany's Excellence Strategy EXC 2181/1 - 390900948 (the Heidelberg STRUCTURES Excellence Cluster and SFB1324 (B05). K. Suzuki was supported by Scientific Research (C) 18K03354 and 19K03557.}


\begin{abstract}
	A general system of $n$ ordinary differential equations coupled with one reaction-diffusion equation, considered in a bounded $N$-dimensional domain, with no-flux boundary condition is studied in a context of pattern formation. Such initial boundary value problems may have different  types of stationary solutions.  
	In our parallel work [\textit{Instability of all regular stationary solutions to reaction-diffusion-ODE systems} (2021)], regular (\textit{i.e.}~sufficiently smooth) stationary solutions are shown to exist, however, all of them are unstable. The goal of this work is to construct discontinuous stationary solutions to general reaction-diffusion-ODE systems and to find sufficient conditions for their stability.
\end{abstract}


\maketitle

\section{Introduction} 

This paper is devoted to analysis of pattern formation in systems coupling diffusing and non-diffusing components. We consider a generic system of $n$ ordinary differential equations coupled with a reaction-diffusion equation 
\begin{nalign}
	\bu_t  &=   \bff(\bu,v), &&x\in\overline{\Omega}, \quad t>0,  \label{eq1}\\
	v_t  &=  \gamma \Delta v+g(\bu,v), &&  x\in \Omega, \quad t>0,
\end{nalign}
where 
\begin{nalign}
	\bu = \bu (x,t) = 
	\left( 
	\begin{array}{c}
		u_1(x,t) \\ 
		\vdots \\
		u_n(x,t)
	\end{array}
	\right) 
	\qquad\text{and}\qquad 
	v = v(x,t),
\end{nalign}
with the Neumann boundary condition \noeqref{Neumann}
\begin{nalign}
	\label{Neumann}
	\partial_{\nu}v  =  0  \qquad \text{for}\quad x\in \partial\Omega, \quad t>0,
\end{nalign}
where $\partial_\nu = \nu\cdot\nabla$ with  $\nu$ denoting  the unit
outer normal vector to $\partial \Omega$. 
We consider arbitrary $C^2$-nonlinearities 
\begin{nalign} 
	\label{nonlinearity}
	\bff=\bff(\bu,v) = 
	\left(
	\begin{array}{c}
		f_1(\bu,v) \\
		\vdots\\
		f_n(\bu,v)
	\end{array}
	\right)
	\qquad\text{and}\qquad  
	g=g(\bu,v),
\end{nalign}
and a bounded open domain $\Omega\subseteq \R^N$ for $N \geq 1$ with $C^2$-boundary $\partial\Omega$. The constant $\gamma>0$ denotes a diffusion coefficient in the reaction-diffusion equation from system~\eqref{eq1}.

Our goal is to identify a class of stationary solutions and to examine their stability. Note that a couple $\UV=\UVx$ is a stationary solution of this problem if it satisfies the relation 
\begin{nalign}
	\label{eq:ImpFun}
	\pmb f \UVx  = 0, \quad x\in \overline{\Omega}
\end{nalign}
and the boundary value problem 
\begin{nalign}
	\label{eq:ImFun2}
	&\gamma\Delta V + g\UV = 0, && x\in \Omega, \\
	&\partial_{\nu} V = 0, && x\in \partial\Omega.
\end{nalign}
Problem \eqref{eq:ImpFun}-\eqref{eq:ImFun2} may have different types of solutions. Particular versions of such a system have been studied, {\it e.g.}~in the  works
\cite{MR3973251, MR3679890, MR3600397, MR3583499, MR3345329, MR3329327, MR3214197, MR3059757, MR3039206, MR2833346,MR3220545,MR684081,MR579554,10780947202173109,MR4213664,MR730252}, 
 where the following two types of stationary solutions $\UVx$ were identified:

\begin{enumerate}
	\item \textit{Regular stationary solutions}, where
	$\pmb U(x) = \pmb k\big(V(x)\big)$ for one $C^2$-function $\pmb k=(k_1, ...,k_n)$ and all $x\in \overline{\Omega}$.  
	\label{it:Type1}
	\item  \textit{Jump-discontinuous stationary solutions} with $\pmb U(x)$ obtained using different branches of solutions with respect to $\pmb U(x)$ to equation~\eqref{eq:ImpFun}. 
\end{enumerate} 

The first type of solutions has been investigated recently in the paper \cite{CMCKS01} showing that all regular stationary solutions  are unstable. This work is devoted to the second type. We construct  a class of stationary solutions with jump-discontinuities to problem \eqref{eq1}-\eqref{Neumann} and we provide sufficient conditions for their stability. The results are formulated for generic nonlinearities \eqref{nonlinearity}, covering the specific nonlinearities considered explicitly in previous works on existence and stability of jump-discontinuous stationary solutions. Such solutions have been investigated initially in construction of solutions for two-components systems 
 \begin{nalign}
 	\varepsilon^2 u_{xx}+f(u,v)=0, \quad v_{xx}+g(u,v)=0,
 \end{nalign}
 with a small parameter $\varepsilon$ (\cite{MR419961,MR684081, MR579554}). The jump-discontinuous solution for $\varepsilon =0$ served  as a lowest order approximation in construction of large amplitude solutions to  the system with sufficiently small $\varepsilon>0$.
Stable discontinuous stationary solutions to a system of one ODE coupled with one reaction-diffusion equation with a degenerated diffusion in one dimensional case have been constructed by Weinberger~\cite{MR730252}. \rev{Novick-Cohen and Pego \cite{novick1991stable} proved a stability of discontinuous steady states for an initial boundary value problem for the equation
\begin{nalign}
	u_{t}=\Delta f(u)+\nu \Delta u_{t}
\end{nalign}
which shares some mathematical properties with reaction-diffusion-ODE systems, see~\cite{perthame2020fast,skrzeczkowski2022fast}.} Discontinuous stationary solutions have been also investigated in \rev{other} specific models from applications (\cite{MR2205561,Kthe2020HysteresisdrivenPF,MR3220545}) showing that hysteresis in the nullclines of the ODE-subsystem may lead to emergence of such solutions in a one-dimensional domain. The approach established in this paper allows constructing a special type of discontinuous stationary solutions to general problem \eqref{eq1}-\eqref{nonlinearity} in a higher dimensional domain and provides their stability under less restrictive assumptions. The theory developed here and in \cite{CMCKS01} can be also applied to  other  reaction-diffusion-ODE models in spatially homogeneous environments arising from applications, \cite{MR3214197,MR3583499,MR3679890, MR3973251,MR3345329,MR2297947,MR684081,MR579554,perthame2020fast,MR2205561,gg}. Moreover, the established approach may be also extended to  reaction-diffusion-ODE models in heterogeneous environments considered~\textit{e.g.}~in the book \cite{MR2191264} and in the papers \cite{MR2527521,10780947202173109,MR4213664, MR3968976, MR4147361,MR4135686}. 

This work is structured in the following way. Main results, stated in the next section, concern the existence of jump-discontinuous stationary solutions (Theorem~\ref{DisExBan}), their linear stability (Theorem \ref{thm:ApplicationSystemStabilityDiscontinuous2}) and nonlinear stability (Theorem~\ref{thm:ApplicationSystemStabilityDiscontinuous}). In Section \ref{sec:DiscontinuousSolution}, we construct jump-discontinuous stationary solutions. Since stability of stationary solutions is obtained using the linearisation procedure, in Section \ref{sec:LinearEquation}, we study stability of zero solution to general linear reaction-diffusion-ODE systems. In Section~\ref{sec:Nonlinear Stability}, we prove that the linear stability implies the nonlinear stability where this classical approach has to be suitably adapted to reaction-diffusion-ODE systems. The proofs of stability Theorems \ref{thm:ApplicationSystemStabilityDiscontinuous2} and \ref{thm:ApplicationSystemStabilityDiscontinuous} are contained in Section \ref{sec:Application}. The obtained theory is illustrated \rev{in} Section \ref{sec:Examples} by numerical simulations of a selected reaction-diffusion-ODE system showing emergence of the constructed patterns.

\subsection*{Notation.} 

By the bold font, e.g.\ $\pmb A, \pmb u,$ we denote either matrices or vector valued functions in order to distinguish them from scalar quantities. 
For $\pmb f$ and $g$ defined in~\eqref{nonlinearity} we set 
\begin{nalign}
	\pmb f_{\pmb u} &= \begin{pmatrix}
		\frac{\partial f_1}{\partial u_1}  & \cdots & \frac{\partial f_1}{\partial u_n}  \\
		\vdots & & \vdots \\
		\frac{\partial f_n}{\partial u_1}  & \cdots & \frac{\partial f_n}{\partial u_n}  \\
	\end{pmatrix}, & 	
	\pmb f_{v} &= 
	\begin{pmatrix}
		\frac{\partial f_1}{\partial v}  \\
		\vdots  \\
		\frac{\partial f_n}{\partial v} \\
	\end{pmatrix}, \;  
	g_{\pmb u} &= 
	\begin{pmatrix}
		\frac{\partial g}{\partial u_1}  & \cdots & \frac{\partial g}{\partial u_n}  \\
	\end{pmatrix}, & 	\; 
	g_{v} &= \frac{\partial g}{\partial v} .
\end{nalign} 
For a matrix $M= (m_{i,j})$, we set $|M| = \sqrt{\sum_{i,j} m_{i,j}^2}$. The product $Y^n = {Y\times \cdots \times Y}$ ($n$-times)  of a given space $Y$ is supplemented with the norm denoted by $\|\pmb  y\|_Y$  (instead of $\| \pmb y \|_{Y^n}$). For $\Omega \subseteq \R^N$, we denote by $\overline{\Omega}$ its topological closure. The symbol  $\sigma(L)$ means the spectrum of a linear operator $\big( L, D(L)\big)$ and $s(L) = \sup\lbrace \Re \lambda : \lambda \in \sigma(L) \rbrace$ is the  spectral bound of $\big( L, D(L)\big)$. The usual Sobolev spaces are denoted by $W^{k,p}(\Omega)$. We use the symbol  $\Delta_\nu$ for the Laplace operator in a bounded domain $\Omega$ with the Neumann boundary condition. It is defined in the usual way via a bilinear form on $W^{1,2}(\Omega)$, and $-\Delta_\nu$ has the eigenvalues $\mu_k$  satisfying $0 = \mu_0 < \mu_1 \leqslant \cdots \leqslant \mu_n \to \infty$. We recall the domain of $-\Delta_\nu$ on $\Lp{}$ with $p\in [1,\infty)$ 
\begin{nalign}
	W^{2,p}_\nu(\Omega) = \{ v\in W^{2,p}(\Omega): \, \partial_{\nu} v = 0 \text{ on } \partial \Omega \},
\end{nalign} 
which is the Sobolev space $W^{2,p}(\Omega)$ supplemented with the Neumann boundary condition. Constants in  estimates below are denoted by the same letter $C$, even if they vary from line to line. Sometimes we shall emphasize dependence of such constants on parameters used in our calculations.

\section{Main results} 

\subsection{Stationary solutions}

The goal of this work is to construct a certain class of non-constant stationary solutions of problem \eqref{eq1}-\eqref{Neumann} and to show their stability. Thus, we deal with a couple  $(\pmb U, V) = \big( \pmb U(x), V(x) \big)$ 
with a vector field and a scalar function
\begin{nalign}
	 \pmb U (x) = 
	\left( 
	\begin{array}{c}
		U_1(x) \\ 
		\vdots \\
		U_n(x)
	\end{array}
	\right) 
	\qquad\text{and}\qquad 
	V = V(x),
\end{nalign}
which is a solution 
to the boundary value problem
\begin{nalign}
	\label{DisProbDef}
	\pmb f(\pmb U,V)&=0,  &&     &&x\in\overline{\Omega},  &&\\
	\gamma \Delta_\nu V+g(\pmb U,V)&=0,  &&  &&x\in \Omega,
\end{nalign}
with arbitrary $C^2$-functions $\pmb f$ and $g$ of the form \eqref{nonlinearity}, with constant $\gamma>0$, and in a bounded domain  
$\Omega\subseteq \R^N $ with a $C^2$-boundary.

\begin{defi}\label{def:stat}
	A pair $\UV = \UVx$ 
	is 
	a {\it weak solution} of problem   \eqref{DisProbDef} 
	if 
	\begin{itemize}
		\item$\pmb U$ is measurable, 
		\item $V\in \W12$,
		\item $g(\pmb U, V) \in \big(\W12\big)^\star$ $($the dual of $\W12)$,    
		\item the equation $\pmb f\big(\pmb U(x),V(x)\big)=0$ is satisfied for almost all $x\in \overline{\Omega}$,
		\item the equation 
		\begin{nalign}
			-\gamma\int_\Omega \nabla V(x)\cdot \nabla \varphi(x)\dx +\int_\Omega 		g\big(\pmb U(x),V(x)\big)\varphi(x)\dx=0 
		\end{nalign}
		holds true for all test functions $\varphi\in \W12$.
	\end{itemize}
\end{defi}

\begin{rem}
	The most direct way to construct such stationary solutions is to solve first equation in \eqref{DisProbDef} with respect to $\pmb U$ and to substitute the result into the second equation to obtain an elliptic boundary value problem. Namely, we may assume that there exists  a $C^2$-function $\pmb k: \R\to \R^n$ such that $\pmb U(x)=\pmb k(V(x))$ for all $x\in \Omega$ and
	\begin{nalign}
		\pmb f\big(\pmb U(x),V(x)\big) = \pmb f\big(\pmb k(V(x)),V(x)\big)=0 \quad \text{for all}\quad x\in \Omega,
	\end{nalign}
	where $V=V(x)$ is a solution of  the elliptic Neumann problem  
	\begin{align}
		\gamma\Delta_\nu V+h(V)=0\qquad   \text{for}\quad x\in \Omega \label{s:h} 
	\end{align}
	with $h(V)=g\big(\pmb k(V),V\big)$. In our parallel work \cite{CMCKS01}, such solutions are called regular and we showed that all of them are unstable. 
\end{rem}

\begin{rem}
	\label{Rem:MatLin}
	Results in this work require that problem \eqref{DisProbDef} has a constant solution, namely,  a constant vector 
	\begin{nalign}
		(\overline{ \pmb U}, \overline{ V}) \in \R^{n+1} \quad\text{such that}\quad  \pmb f(\overline{\pmb U}, \overline{V}) = 0  \quad\text{and} \quad g(\overline{\pmb U}, \overline{V}) = 0.
	\end{nalign} In the following, we use the constant coefficient matrices
		\begin{nalign}
		\label{ass:Matrix}
		\pmb A_0 = \pmb f_{\pmb u}\left( \overline{\pmb U}, \overline{ V} \right),\quad  \pmb 	B_0 = \pmb f_{v}\left( \overline{\pmb U}, \overline{ V} \right), \quad \pmb C_0 =g_{\pmb u}\left( \overline{\pmb U}, \overline{ V} \right), \quad d_0 = g_v\left( \overline{\pmb U}, \overline{ V} \right).
	\end{nalign}	
	On the other hand, if $\UV \in \Li^{n+1}$ is a non-constant  solution to problem~\eqref{DisProbDef} we introduce the matrices
	\begin{align}
		\retainlabel{eq:ReactionDiffusionABCDMatrices}
		\pmb A =  \pmb f_{\pmb u} ({\pmb U},{ V}), \quad  \pmb B =  \pmb f_{v} ({\pmb U},{ V}), \quad 
		\pmb C =  g_{\pmb u} ({\pmb U},{ V}), \quad
		d =   g_{v} ({\pmb U},{ V})
	\end{align}
	with coefficients from $\Li$.
\end{rem}

\subsection{Existence  of discontinuous stationary solutions }

We begin by imposing assumptions required in our construction of solutions to  boundary value problem~\eqref{DisProbDef} which are not regular, \textit{i.e.}\ which are not \rev{necessarily} $C^2$-functions. 

\begin{assumption}
	\label{ass:StatSol}
	The vector $(\overline{\pmb U}, \overline{V})\in \R^{n+1}$ is a constant solution of problem \eqref{DisProbDef} with the corresponding matrices \eqref{ass:Matrix} satisfying
	\begin{nalign}
		\label{Sineq}
		\label{Impil}
		\det \pmb A_0 \neq 0 \quad \text{and} \quad
		\frac{1}{\det  \pmb A_0 }\det\begin{pmatrix}
			\pmb A_0 &    \pmb B_0 \\
			\pmb C_0  &  d_0
		\end{pmatrix} \neq \gamma\mu_k,
	\end{nalign}
	for each eigenvalue $\mu_k$ of $-\Delta_\nu$.	
\end{assumption}

Notice that if $\det \pmb A_0 \neq 0$, by the Implicit mapping theorem, there exist an open neighbourhood $\V \subseteq \R$ of $\overline{ V}$ and a function $\pmb k \in C^2(\V,\R^{n})$ such that
\begin{nalign}
	\pmb k(\overline{V}) = \overline{\pmb U} \quad \text{and} \quad \pmb f(\pmb k(w),w) = 0\quad \text{for all } w\in \V.
\end{nalign}
In the following assumption,  the equation $\pmb f(\pmb U, V) = 0$ is required to have  \textit{two different branches of solutions} with respect to $\pmb U$ on a common domain $\V$.

\begin{assumption}
	\label{ass:DiscontinuousStationaryTwoBranches}
	Let $\big(\overline{ \pmb U}, \overline{ V}\big)$ be a constant  solution of problem \eqref{DisProbDef}. We assume that there exists an open set $\V \subseteq \R$ and $\pmb k_1, \pmb k_2 \in C^2(\V, \R^n)$ such that
	\begin{itemize}
		\item $ \overline{ V} \in \V$,
		\item $\overline{ \pmb U} = \pmb k_1 (\overline{ V}), \quad\pmb k_1 (\overline{ V} ) \neq \pmb k_2(\overline{ V})$, 
		\item $\rev{\pmb f\big(\pmb k_1(w),w\big) = \pmb f\big(\pmb k_2(w),w\big)} = 0 \ \text{for all} \ w\in \V.$
	\end{itemize}	
\end{assumption}

Such two branches allow us to construct discontinuous solution of problem \eqref{DisProbDef}.

\begin{theorem}[Existence of solutions]
	\label{DisExBan}
	Assume that
	\begin{itemize}
		\item problem \eqref{DisProbDef} with $\gamma>0$ has a constant solution $(\overline{  \pmb U}, \overline{ V}) \in \R^{n+1}$ satisfying Assumption \ref{ass:StatSol},
		\item \rev{Assumption} \ref{ass:DiscontinuousStationaryTwoBranches} holds true.
	\end{itemize}
	\rev{There is $\varepsilon_0 >0$ such that for all $\varepsilon \in (0, \,\varepsilon_0)$ there exists $\delta >0$ such that for arbitrary open subset $\Omega_1\subseteq\Omega$ with $\Omega_2 = \Omega \setminus \overline{\Omega_1}$ satisfying $|\Omega_2|< \delta$,} problem \eqref{DisProbDef} has a weak solution in the sense of Definition \ref{def:stat} with the following properties:
	\begin{itemize}
		\item $(\pmb U, V)\in L^\infty(\Omega)^n \times \W2p$ for each \rev{$p\in[2,\infty) \cap (N/2,\infty)$},
		\item $V=V(x)$ is a weak solution to problem,
		\begin{nalign}
			\label{eq:DinsontinousSolutionsVEquation}
			\gamma\Delta_\nu V + g(\pmb U,V) = 0 \quad \text{for} \quad x\in \Omega,
		\end{nalign}
		where
		\begin{nalign}
			\label{Uform}
			\pmb U (x) = \begin{cases}
				\pmb k_1\big(V(x)\big), \quad x\in \Omega_1, \\
				\pmb k_2\big(V(x)\big), \quad x\in \Omega_2,
			\end{cases}
		\end{nalign} 
		satisfies $\pmb f\big(\pmb U(x),V(x)\big) = 0$ for almost all $x\in \overline{\Omega}$.
		\item the couple $\UV$ \rev{stays close to the points 
		$\big(\overline{ \pmb U}, \overline{ V}\big)$  and 
		$\big( \pmb k_2 (\overline{ V}), \overline{ V}\big)$
		in the following sense {\rm (}recall that $\overline{ \pmb U}= \pmb k_1 (\overline{ V})${\rm )}:}
		\begin{nalign}
			\label{DisVar}
			\|V - \overline{ V} \|_\Li < \varepsilon \quad \text{and} \quad \|\pmb U -  \pmb k_1 (\overline{ V}) \|_{L^\infty(\Omega_1)} + \|\pmb U -  \pmb k_2 (\overline{ V}) \|_{L^\infty(\Omega_2)}  < C\varepsilon
		\end{nalign}
		for a constant $C = C(\pmb f,  g)$. 
	\end{itemize}
\end{theorem}

\begin{rem}
	\label{rem:Conti}
	In particular, by the Sobolev embedding, we have $V\in C(\oo)$. On the other hand, the function $\pmb U = \pmb U(x)$ given by formula \eqref{Uform} is discontinuous when $\pmb k_1(w) \neq \pmb k_2(w)$ for all $w \in \V$ which holds true when $\varepsilon>0$ in inequalities~\eqref{DisVar} is sufficiently small.
\end{rem}

\subsection{Stability of discontinuous stationary solutions}

Let $\big( \pmb U, V \big) \in \Li^{n+1}$ be a stationary solution of problem \eqref{eq1}-\eqref{nonlinearity}. We linearise problem \eqref{eq1}-\eqref{nonlinearity} around $\UV$ by introducing the new functions
\begin{nalign}
	\pmb \varphi = \pmb u - \pmb U \quad \text{and} \quad \psi = v - V
\end{nalign}
which satisfy the system
\begin{nalign}
	\label{eq:ReactionDiffusionLinearisedProblemWithRest}
	\frac{\partial}{\partial t}\begin{pmatrix} \pmb \varphi \\ \psi \end{pmatrix} = \begin{pmatrix} 0 \\ \gamma\Delta_\nu \psi \end{pmatrix} + \begin{pmatrix} \pmb A & \pmb B \\ \pmb C & d \end{pmatrix} \begin{pmatrix} \pmb \varphi \\ \psi \end{pmatrix} + \begin{pmatrix} \pmb R_1(\pmb \varphi, \psi) \\ R_2(\pmb \varphi, \psi) \end{pmatrix}
\end{nalign}
with the matrices \eqref{eq:ReactionDiffusionABCDMatrices} and with the usual remainders $\pmb R_1$, $R_2$ obtained from the Taylor expansion. We also impose the initial datum
\begin{nalign}
	\label{eq:LinIni}
	\pmb \varphi(\cdot, 0) = \pmb u_0 - \pmb U, \quad \psi(\cdot,0) = v_0 - V,
\end{nalign}
for some $\pmb u_0\in \Li^n$, $v_0\in C(\oo)$.

We begin by an analysis of a linear problem and say that $\UV$ is {\it linearly exponentially  stable} if the zero solution to system \eqref{eq:ReactionDiffusionLinearisedProblemWithRest} with the remainders $\pmb R_1 = R_2 = 0$ is exponentially stable. Our stability result requires additional assumptions imposed on the constant solution $(\overline{ \pmb U}, \overline{V})\in \R^{n+1}$ used in Theorem \ref{DisExBan}. 

\begin{assumption}
	\label{ass:LinearStability}
	Matrices \eqref{ass:Matrix} satisfy 
	\begin{nalign}
		s(\pmb A_0)<0 \quad \text{and} \quad s\begin{pmatrix} \pmb A_0 & \pmb B_0 \\ \pmb C_0 & d_0 \end{pmatrix} < 0,
	\end{nalign} 
	which means that both matrices have all eigenvalues with strictly negative real parts. 
\end{assumption}

\begin{assumption} 
	\label{ass:LinearStabilityGamma}
	For fixed $p\in(1,\infty)$, the diffusion coefficient  satisfies $\gamma \geqslant  \gamma_0$ where $\gamma_0 = \gamma_0(\pmb A_0, \pmb B_0, \pmb C_0, d_0, |\Omega|, p)>0$ is a constant provided by Lemma~\ref{thm:LinearSystemEstimatesForUVByFG}, below.
\end{assumption}

\begin{rem}
	Assumptions \ref{ass:LinearStability} and \ref{ass:LinearStabilityGamma} are needed to show stability of a zero solution of the following linear system with constant coefficients
	\begin{nalign}
		\label{eq:BCon}
		\frac{\partial}{\partial t}\begin{pmatrix} \pmb \varphi \\ \psi \end{pmatrix} = \begin{pmatrix} 0 \\ \gamma\Delta_\nu \psi \end{pmatrix} + \begin{pmatrix} \pmb A_0 & \pmb B_0 \\ \pmb C_0 & d_0 \end{pmatrix} \begin{pmatrix} \pmb \varphi \\ \psi \end{pmatrix}.
	\end{nalign}
	Obviously, for $s\begin{spm} \pmb A_0 & \pmb B_0 \\ \pmb C_0 & d_0, \end{spm} < 0$, zero is a stable solution to this system with $\gamma = 0$. If $s(\pmb A_0)>0$, zero solution of system \eqref{eq:BCon} with arbitrary $\gamma >0$ is unstable which results from the characterization of the spectrum of the linear operator in \eqref{eq:BCon} obtained in \cite[Thm. 4.5]{CMCKS01} and recalled below in equation \eqref{thm:SpectrumOfOperatorL}. Thus, for stability results, we have to impose the assumption $s(\pmb A_0)<0$. Finally, $\gamma>0$ has to be large enough in Lemma~\ref{thm:LinearSystemEstimatesForUVByFG} in order to obtain an exponential stability of solutions to system \eqref{eq:BCon}. 
\end{rem}

\begin{theorem}[Linear stability]
	\label{thm:ApplicationSystemStabilityDiscontinuous2}
	Let the assumptions of Theorem \ref{DisExBan} holds true. Fix $p\in (1,\infty)$. The stationary solution $\UV$ constructed in Theorem \ref{DisExBan} is linearly exponentially stable in $\Lp{}^{n+1}$ if moreover
	\begin{itemize}
		\item Assumptions \ref{ass:LinearStability} and \ref{ass:LinearStabilityGamma} are valid,
		\item $\varepsilon>0$ and $|\Omega_2|>0$ are small enough in Theorem \ref{DisExBan},
		\item the following inequality is satisfied
		\begin{nalign}
			\label{eq:SecBranch}
			s\big( \pmb f_{\pmb u}  (\pmb k_2(\overline{V}), \, \overline{ V} ) \big) < 0.
		\end{nalign}
	\end{itemize}
\end{theorem}

Theorem \ref{thm:ApplicationSystemStabilityDiscontinuous2} is proved in Section \ref{sec:Application} and it is a direct consequence of Theorem~\ref{thm:ResEsNonCon} below, containing a spectral analysis of a linear operator which appears in system \eqref{eq:ReactionDiffusionLinearisedProblemWithRest}.

\begin{rem} 
	\label{thm:AutoPoin}
	The solution $(\pmb U, V) \in \Li^n \times C(\oo)$ constructed in Theorem~\ref{DisExBan} can not be stable in $\Lp{}^{n+1}$ if
	\begin{nalign}
		\text{either} \quad s\big(\pmb f_{\pmb u}( \pmb k_1(\overline{V}), \, \overline{ V}) \big)>0\quad \text{or} 	\quad  s\big(\pmb f_{\pmb u} (\pmb k_2(\overline{V}), \, \overline{ V})\big)>0,
	\end{nalign}
	and if $\varepsilon >0$, $|\Omega_2| >0$ are sufficiently small. 	
	This follows immediately from the fact that eigenvalues of the matrix $\pmb f_{\pmb u} \big(\pmb U(x), V(x)\big)$ belong to $\sigma(\L_p)$ for almost every $x\in \Omega$, (see formulas \eqref{thm:SpectrumOfOperatorL} and \eqref{eq:SpecEq} below) and from inequalities \eqref{DisVar} ensuring that $\pmb f_{\pmb u} (\pmb U(x), V(x))$ stays close to $\pmb f_{\pmb u} (\pmb k_1(\overline{  V}), \overline{ V})$ on $\Omega_1$ and close to $\pmb f_{\pmb u} (\pmb k_2(\overline{  V}), \overline{ V})$ on $\Omega_2$. 
\end{rem}

\begin{theorem}[Nonlinear stability]
	\label{thm:ApplicationSystemStabilityDiscontinuous}
	Let the assumptions of Theorem \ref{thm:ApplicationSystemStabilityDiscontinuous2} hold true. Assume, moreover,
	\begin{nalign}
		\label{eq:Dcon}
		g_v\big( \pmb k_1 (\overline{V}), \overline{V}\big) < 0 \quad \text{ and } \quad 	 g_v\big( \pmb k_2 (\overline{V}), \overline{V}\big) < 0.
	\end{nalign}
	Then the stationary solution $\UV$ constructed in Theorem \ref{DisExBan} is exponentially asymptotically stable in the Lyapunov sense in~$\Li^{n+1}$ provided $\varepsilon>0$ in inequalities \eqref{DisVar} is small enough.
\end{theorem}

\begin{rem}
	Theorem \ref{thm:ApplicationSystemStabilityDiscontinuous} is a direct consequence of Theorem \ref{thm:NonlinearStability} containing a proof of the so-called \textit{linearisation principle}, \rev{where} a linear exponential stability of zero solution of system \eqref{eq:ReactionDiffusionLinearisedProblemWithRest} implies a nonlinear exponential stability. To adapt this well known procedure to reaction-diffusion-ODE systems, we have to deal with certain obstacles. For example, the linearised operator may have a continuous spectrum without a spectral gap, see equation  \eqref{thm:SpectrumOfOperatorL} and we work in $\Li$ where the linearised operator does not generate strongly continuous semigroup of linear operators.  Here, the nonlinear stability in $\Li$ is obtained by using the linear exponential stability in $\Lp{}$ with $p\in (1, \infty)$ from Theorem~\ref{thm:ApplicationSystemStabilityDiscontinuous2}.
\end{rem} 

\begin{rem}
	\label{thm:3Branches}
	Other discontinuous stationary solutions can be constructed under following more general version of Assumption \ref{ass:DiscontinuousStationaryTwoBranches}, where we postulate an existence of a set $\V\subseteq \R$ and different branches $\pmb k_1, \cdots, \pmb k_J \in C(\V, \R^n)$ of solutions to equation $\pmb f( \pmb U, V) = 0$.  Then for an arbitrary decomposition 
	\begin{nalign}
		\Omega = {\bigcup_{i\in \lbrace 1,\cdots, J \rbrace} \Omega_i}, \quad \Omega_i \cap \Omega_j =\emptyset \ \ \text{for} \ \ i\neq j \quad \text{and} \quad |\Omega_i| <\delta \ \ \text{for} \ \ i\in\lbrace 2, \cdots, J\rbrace,  
	\end{nalign} 
	we can construct a discontinuous stationary solution of the form
	\begin{nalign}
		\label{Uformn}
		\pmb U (x) = \begin{cases}
			\pmb k_1\big(V(x)\big), \quad x\in \Omega_1, \\
			\qquad \vdots \\
			\pmb k_J\big(V(x)\big), \quad x\in \Omega_J,
		\end{cases}
	\end{nalign} 
	using the same reasoning as in Theorem \ref{DisExBan}. For sufficiently large $\gamma>0$, this solution is linearly stable if $s\big( \pmb f_{\pmb u}  (\pmb k_i(\overline{V}), \, \overline{ V} ) \big) < 0$ for each $i\in\lbrace 1, \cdots, J\rbrace$ and nonlinearly stable if, moreover,  $g_v( \pmb k_i (\overline{V}), \overline{V}) < 0$ for each $i\in\lbrace 1, \cdots, J\rbrace $.
\end{rem}

\begin{rem}
	For $n=1$, (\textit{i.e.}~when one ODE is coupled with one reaction-diffusion equation), the stationary solution $(U,V)$ constructed in Theorem~\ref{DisExBan} is linearly exponentially stable in $\Lp{}^{n+1}$ for each $p\in (1, \, \infty)$ and every $\gamma>0$ (not necessarily large as required in Assumption \ref{ass:LinearStabilityGamma}) under the following conditions
	\begin{nalign}
		\label{Stab2}
		s\Big(  f_{ u}\big(U(\cdot),V(\cdot)\big)\Big)<0, \quad s\begin{pmatrix} 
			f_{ u}\big(U(\cdot),V(\cdot)\big) & f_{ v}\big(U(\cdot),V(\cdot)\big) \\ g_{ u}\big(U(\cdot),V(\cdot)\big) & g_{ v}\big(U(\cdot),V(\cdot)\big)
		\end{pmatrix}<0
	\end{nalign}
	and
	\begin{nalign}
		s\big(g_v(U(\cdot), V(\cdot))\big) <0,
	\end{nalign}
	which result from Assumption \ref{ass:LinearStability} and inequalities \eqref{eq:Dcon} for sufficiently small $\varepsilon>0$ and $|\Omega_2|>0$. This fact is an immediate consequence of Proposition \ref{thm:LinearStabilitystantNoncoefficientsN1}. 
\end{rem}

\begin{rem}
	The $L^\infty$-stability in Theorem \ref{thm:ApplicationSystemStabilityDiscontinuous} requires from initial conditions \rev{of} problem \eqref{eq1}-\eqref{nonlinearity} to have the same discontinuities as considered stationary solutions. A different notion of stability (so-called $(\varepsilon,A)$-stability) is used in the papers \cite{MR3973251, MR3583499,MR4213664,MR730252} on particular reaction-diffusion models.
	In that approach,  initial conditions can be continuous functions and neighborhoods of discontinuities of stationary solutions are excluded from the stability analysis.
	We do not apply that approach in this work, because it requires from the considered model to have an invariant region.
\end{rem}

\section{Construction  of discontinuous stationary solutions} 

\label{sec:DiscontinuousSolution}

The proof of Theorem \ref{DisExBan} is based on the following preliminary results. 

\begin{lemma}
	\label{DisLemma1} 
	Fix $p\in [2, \infty)$ and let $\gamma>0$ be arbitrary. For each $b\in \R$ such that $b\neq \gamma\mu_k$ for all eigenvalues $\mu_k$ of $-\Delta_\nu$, and for every $f\in L^p(\Omega)$ the problem
	\begin{nalign}
		\label{DisLinearEquation}
		\gamma \Delta_\nu v + bv & = f \quad \text{for}\quad x\in \Omega,
	\end{nalign}
	has a unique weak solution $v\in \W1p$ satisfying
	\begin{nalign}
		\label{SobNormEst}
		\|v\|_\W1p \leqslant C(b,\gamma) \|f\|_\Lp{}.
	\end{nalign}
	If $\partial \Omega$ is of the class $C^2$, then $v\in \W2p$ and there exists a constant $C= C(b,p,\gamma)>0$ such that 
	\begin{nalign}
		\label{SobNormEstImproved} 		
		\|v\|_\W2p \leqslant C \|f\|_\Lp{}.
	\end{nalign}
\end{lemma}

\begin{proof}
	In the case of $p=2$ and $b <0$, this is a well known result on the existence and $W^{2,2}$-regularity of solution to problem \eqref{DisLinearEquation}, see \textit{e.g.}\ \cite[Thm. 9.26]{MR2759829}. In fact, we can extend this result to all $b\neq \gamma \mu_k$ because the solution to problem \eqref{DisLinearEquation} is given by the explicit formula
	\begin{nalign}
		v = \sum_{k=0}^{\infty} \frac{1}{\rev{-\gamma\mu_k} + b}\langle f, \Phi_k\rangle_{\Lp{2}} \Phi_k,
	\end{nalign}
	where $\Phi_k$ is an eigenfunction corresponding to the eigenvalue $\mu_k$.
		
	To prove this Lemma for $p>2$, we rewrite the equation in the form 
	\begin{nalign}
		\gamma \Delta_\nu v - v & = f - (b+1)v \quad  \text{for} \quad x\in \Omega.
	\end{nalign}
	By using the $L^p$-elliptic regularity (see \textit{e.g.}~\cite[Thm. 3.1.3]{MR3012216}), the Sobolev embedding $\W22 \subseteq \Lp{} $ for each $p\in (1,\infty)$ if \rev{$N\leqslant 4$} and for $p\in \big(2,\, 2N/(N-4)\big)$ if $N>4$, as well as inequality \eqref{SobNormEstImproved} with $p=2$, we obtain
	\begin{nalign}
		\label{PIneq}
		\|v\|_\W2p &\leqslant C \big(\|f\|_\Lp{} + \|v\|_\Lp{}\big) \\
		&\leqslant C \big(\|f\|_\Lp{} + \|v\|_\W22\big)   \leqslant C\|f\|_\Lp{}.
	\end{nalign}
	
	In order to obtain the desired estimate for the whole range of $p\in [2, \infty)$ it suffices to apply a bootstrap argument involving inequality \eqref{PIneq} together with the Sobolev embedding $\W2p\subseteq\Lp{q}$ for $q<Np/(N-2p)$.
\end{proof}

We obtain discontinuous stationary solutions of problem \eqref{DisProbDef} from the following lemma which is a consequence of the Banach fixed point argument.

\begin{lemma}
	\label{DisTheorem1}
	Let $p>N/2$ with $p\geqslant 2$, $\gamma>0$, and $h_1, h_2 \in C_b^2(\R)$ {\rm (}\textit{i.e.}\ all derivatives up to order two are bounded{\rm )}. Assume that for some $\overline{v} \in \R$ we have  
	\begin{nalign}
		h_1(\overline{v}) = 0 \quad \text{and} \quad h_1'(\overline{v}) \neq \gamma\mu_k
	\end{nalign} 
	for each eigenvalue $\mu_k$ of $-\Delta_\nu$. \rev{For arbitrary open subset $\Omega_1\subseteq \Omega$ with $\Omega_2 = \Omega\setminus\overline{\Omega_1}$} 
	one can find $\varepsilon_0 >0$ such that for each $\varepsilon\in(0, \; \varepsilon_0)$ there exists $\delta>0$ such that if $|\Omega_2|\leqslant \delta$ then problem 
	\begin{nalign}
		\label{DisDefSingle}
		\gamma\Delta_\nu v + h_1(v) \mathbbm{1}_{\Omega_1} + h_2(v) \mathbbm{1}_{\Omega_2} = 0 \qquad for \quad x\in \Omega,
	\end{nalign}
	has a weak solution $v\in \W2p$ satisfying $\|v-\overline{v}\|_\W2p\leqslant \varepsilon$. 
\end{lemma}

\begin{proof}
	We denote by $v=\big(\gamma \Delta_\nu  + h_1'(\overline{v})\big)^{-1}f$ the solution of problem \eqref{DisLinearEquation} with $b = h_1'(\overline{v})$ which by Lemma~\ref{DisLemma1} satisfies the estimate
	\begin{nalign}
		\label{OperatorLNormW22Estimation}
		\left\| \big(\gamma \Delta_\nu  + h_1'(\overline{v})\big)^{-1} f \right\|_\W2p \leqslant C\|f\|_\Lp{}
	\end{nalign}
	for each $p\geqslant2$ and all $f\in \Lp{}$. 
	
	Replacing $v$ by $v+\overline{v}$ and $h_i(v)$ for $i\in \{ 1,2\}$ by $h_i(v+\overline{v})$ we may assume without loss of generality that $\overline{v} = 0$. We write equation \eqref{DisDefSingle} in the equivalent form
	\begin{nalign}
		\label{T0}
		v &= \big(\gamma \Delta_\nu  + h_1'(0)\big)^{-1} \Big(\big(  h'_1(0) v - h_1(v)  \big) \mathbbm{1}_{\Omega_1} \Big) \\ &+ \big(\gamma \Delta_\nu  + h_1'(0)\big)^{-1} \Big(\big(h_1'(0)v - h_2(v) \big) \mathbbm{1}_{\Omega_2}\Big) \\
		&= \mathcal{T}_1(v) + \mathcal{T}_2(v).
	\end{nalign}
	Our goal is to show that the right-hand side of this equation defines a contraction on the ball
	\begin{nalign}
		B_\varepsilon(0) = \lbrace v \in \W2p: \|v\|_\W2p \leqslant \varepsilon \rbrace
	\end{nalign}
	provided $|\Omega_2|$ and $\varepsilon >0$ are sufficiently small. 
	
	First, we deal with the operator $\mathcal{T}_1$. Since \rev{$h_1 \in C_b^2(\R)$} and $h_1(0) = 0$ by the Taylor expansion we get
	\begin{nalign}
		\left|h_1'(0)v(x) - h_1\big(v(x)\big)\right|=\left|h_1\big(v(x)\big) - h_1(0) - h_1'(0)v(x)\right|\leqslant \frac{1}{2} \|h_1''\|_\LiR \big|v(x)\big|^2.
	\end{nalign}
	Hence, by estimate \eqref{OperatorLNormW22Estimation} and the Sobolev embedding $\W2p \subseteq \Li$ for $p>N/2$, there exists a constant $C >0$ such that for each $v\in B_\varepsilon(0)$ we have
	\begin{nalign}
		\label{OperatorT1BTBEstimation}
		\| \mathcal{T}_1(v) \|_\W2p \leqslant  C \|v\|_\Li^2 \leqslant C\varepsilon^2 .
	\end{nalign}
	Analogously, by the Taylor expansion,
	\begin{nalign}
		\Big|h_1'(0)v(x) &- h_1\big(v(x)\big)-\Big(h_1'(0)w(x) - h_1\big(w(x)\big)\Big) \Big| \\ 
		& \leqslant \left| h_1\big(v(x)\big) - h_1\big(w(x)\big) +h_1'\big(w(x)\big)\big(v(x)-w(x)\big)\right|  
		\\ &+ \left|\big(h_1'\big(w(x)\big) - h_1'(0) \big)\big(v(x)-w(x)\big)\right| \\
		& \leqslant \frac{1}{2}\|h_1'' \|_\LiR \big|v(x)-w(x)\big|^2  + \|h_1''\|_\LiR \big|w(x)\big|\big|v(x)-w(x)\big|. 
	\end{nalign} 
	Hence, as in inequalities \eqref{OperatorT1BTBEstimation}, for $v,w\in B_\varepsilon(0)$, we obtain
	\begin{nalign}
		\label{OperatorT1Contraction}
		\|\mathcal{T}_1(v) - \mathcal{T}_1(w)\|_\W2p \leqslant C\varepsilon\|h_1''\|_\LiR  \|v-w\|_\W2p.
	\end{nalign}
	Here, we choose $\varepsilon_0>0$ such that for all $\varepsilon\in (0,\varepsilon_0)$ we have 
	\begin{nalign}
		C\varepsilon^2 \leqslant \varepsilon/2 \quad \text{and} \quad C\varepsilon\|h_1''\|_\LiR < 1.
	\end{nalign}
	
	In the case of the operator $\mathcal{T}_2$, we begin by the following inequality 
	\begin{nalign}
		\label{OperatorT2PointwiseEstimate}
		\left|h_1'(0)v(x) - h_2\big(v(x)\big)\right| \leqslant \|h_1'\|_\LiR \big|v(x)\big| + \|h_2\|_\LiR.
	\end{nalign}
	Thus, by estimate \eqref{OperatorT2PointwiseEstimate} and the embedding $\W2p \subseteq \Li$ \rev{for $p>N/2$,} for each $v\in B_\varepsilon(0)$ we get
	\begin{nalign}
		\label{eq:DiscontinuesSolutionsOperatorT2BTB}
		\|\mathcal{T}_2(v)\|_\W2p &\leqslant C \big( \|h_1'\|_\LiR \| v \mathbbm{1}_{\Omega_2}\|_\Lp{} + \|h_2\|_\LiR \|\mathbbm{1}_{\Omega_2}\|_\Lp{} \big) \\
		& \leqslant C|\Omega_2|^{\frac{1}{p}}\big(   \|v\|_\W2p + 1 \big).
	\end{nalign}
	Analogously, we have 
	\begin{nalign}
		\label{OperatorT2DifferencePointiwseEstimate}
		\Big| h_1'(0)v(x) &- h_2\big(v(x)\big) - \big(h_1'(0)w(x) - h_2\big(w(x)\big) \big) \big|\\  & \leqslant (\|h_1'\|_\LiR+\|h_2'\|_\LiR) \big(v(x)-w(x)\big).
	\end{nalign}
	Thus, by inequality \eqref{OperatorT2DifferencePointiwseEstimate},  as in estimate \eqref{eq:DiscontinuesSolutionsOperatorT2BTB}, for $v,w\in B_\varepsilon(0)$, we obtain
	\begin{nalign}
		\label{OperatorT2Contraction}
		\|\mathcal{T}_2(v) - \mathcal{T}_2(w)\|_\W2p & \leqslant  C(\|h_1'\|_\LiR, \|h_2'||_\LiR) \|(v-w)\mathbbm{1}_{\Omega_2}\|_\Lp{} \\
		& \leqslant C |\Omega_2|^{\frac{1}{p}} \|v-w\|_\W2p.
	\end{nalign}
	Here, we choose $\delta >0$ such that for $\Omega_2 \subseteq \R^N$ satisfying $|\Omega_2|<\delta$ we have 
	\begin{nalign}
		C\big( \varepsilon + 1 \big)|\Omega_2|^{\frac{1}{p}} \leqslant \varepsilon/2 \quad \text{and} \quad C |\Omega_2|^{\frac{1}{p}} < 1.
	\end{nalign}
	Estimates \eqref{OperatorT1BTBEstimation}, \eqref{OperatorT1Contraction}, \eqref{OperatorT2PointwiseEstimate} and \eqref{OperatorT2Contraction} together with the Banach fixed point argument complete the proof of the existence of a solution to equation \eqref{T0} in $B_\varepsilon(0)$.
\end{proof}

\begin{proof}[Proof of Theorem \ref{DisExBan}] 
	Let us choose $\varepsilon >0$ so small that $\left[\overline{V} - \varepsilon, \overline{V} + \varepsilon\right] \subseteq \V$. First, for $i\in \{ 1, 2\}$ we restrict $\pmb k_i$ to $\left[\overline{V} - \varepsilon, \overline{V} + \varepsilon\right]$ and then we extend them in arbitrary way to $\tilde{\pmb k}_i \in C^2_b(\R)$ such that $\tilde{\pmb k}_i\big|= \pmb k_i$ on $\left[\overline{V} - \varepsilon, \overline{V} + \varepsilon\right]$. Now we define \rev{$h_i(V) = g\big( \tilde{\pmb k}_i(V),V \big)$} for $i\in \{ 1,2\}$ which satisfy assumptions of Lemma ~\ref{DisTheorem1}. Obviously $h_i \in C^2(\R, \R)$ and $h_1(\overline{V}) = 0$. 
	Applying the identity for the determinants from equation \eqref{eq:DetIden} and using \rev{Assumption \ref{ass:StatSol}} we obtain
	\begin{nalign}
		h_1'(\overline{ V}) &=  g_{\pmb u}\big(\widetilde{\pmb k_1}(\overline{ V}), \overline{ V}\big)\widetilde{\pmb k_1}'(\overline{ V})  + g_{v}\big(\widetilde{\pmb k_1}(\overline{ V}), \overline{ V}\big) \\
		& = \rev{- g_{\pmb u}\big(\widetilde{\pmb k_1}(\overline{ V}), \overline{ V}\big) \pmb f_{\pmb u}\big(\widetilde{\pmb k_1}(\overline{ V}), \overline{ V}\big)^{-1}\pmb f_{v}\big(\widetilde{\pmb k_1}(\overline{ V}), \overline{ V}\big)   + g_{v}\big(\widetilde{\pmb k_1}(\overline{ V}), \overline{ V}\big)}\\
		&= \rev{\frac{1}{\det \left( \pmb f_{\pmb u}(\overline{\pmb U},\overline{ V})\right)}\det\begin{pmatrix}
			\pmb f_{\pmb u} (\overline{\pmb U},\overline{ V}) &   \pmb f_{v} (\overline{\pmb U},\overline{ V}) \\
			g_{\pmb u} (\overline{\pmb U},\overline{ V})  &   g_{v} (\overline{\pmb U},\overline{ V}) 
		\end{pmatrix} \neq \gamma\mu_k.}
	\end{nalign}
	Hence for each $p\in[2,\infty)$, by Lemma \ref{DisTheorem1}, we obtain a solution $V\in \W2p$ to problem \eqref{DisDefSingle} which satisfy $\| V-\overline{V}\|_\W2p \leqslant \varepsilon$ for arbitrary small $\varepsilon>0$ provided $|\Omega_2|>0$ is sufficiently small. However by the Sobolev embedding with $p>N/2$, we have
	\begin{nalign}
		\|V-\overline{V}\|_\Li \leqslant C\|V - \overline{ V}\|_\W2p.
	\end{nalign} 
	Thus, if $\varepsilon >0$ is small enough we obtain $\tilde{\pmb k}_i\big(V(x)\big) = \pmb k_i\big(V(x)\big)$ for all $x\in \Omega$. This completes the construction of solution to problem \eqref{DisProbDef}.
\end{proof}

\section{Linear stability} 

\label{sec:LinearEquation}

\subsection{Linearised operator}\label{sec:LinearizedSpectra} 
In order to apply a linearisation procedure, first, we study a stability of the zero solution to the following general linear reaction-diffusion-ODE system
\begin{nalign}
	\label{eq:LinearProblemDefinition}
	\pmb \varphi_t  &=   \bA  \pmb \varphi+\bB \psi &     \text{for}\quad &x\in\overline{\Omega}, \quad t>0, &&
	\\  \psi_t  &=    \gamma\Delta_\nu \psi+\bC   \pmb \varphi+d \psi&  \text{for}\quad &x\in \Omega, \quad t>0,
\end{nalign}
with general matrices 
\begin{nalign}
	\label{eq:LinearSystemCoefficientMatrixDefinition}
	\bA &= \pmb A(x) =
	\begin{pmatrix}
	a_{11}(x)&\dots&a_{1n}(x)\\
	\vdots&\ddots&\vdots\\
	a_{n1}(x)&\dots&a_{nn}(x)
	\end{pmatrix},
	&\bB&= \pmb B(x) =
	\begin{pmatrix}
	b_{1}(x)\\
	\vdots\\
	b_{n}(x)\\
	\end{pmatrix},
	\\
	\bC &= \pmb C(x) = 
	\begin{pmatrix}
	c_{1}(x)&\dots&c_{n}(x)
	\end{pmatrix}, 
	&d &= d(x),
\end{nalign}
with arbitrary coefficients (not necessarily as those in equation \eqref{eq:ReactionDiffusionABCDMatrices})
\begin{nalign}
	\label{abcd}
	a_{ij}, \; b_i, \; c_i, \; d \in L^\infty(\Omega), \quad \text{for} \; i,j \in \lbrace 1, \cdots, n \rbrace.
\end{nalign}

First, following our parallel work \cite{CMCKS01}, we recall  properties of the operator  defined by the formula
\begin{nalign}
	\label{eq:LinearSystemOperatorLDefinition}
	\L_p \begin{pmatrix}
		\pmb \varphi \\ \psi
	\end{pmatrix} \equiv \begin{pmatrix}
		\pmb A \pmb \varphi+ \pmb B \psi \\ \gamma\Delta_\nu \psi + \pmb C \pmb \varphi + d \psi
	\end{pmatrix} \quad \text{with} \quad D(\L_p) = L^p(\Omega)^n \times W^{2,p}_\nu(\Omega)
\end{nalign} 
for each $p\in (1,\, \infty)$. Here, as usual, we denote by $\sigma(\L_p)$ its spectrum and by $s(\L_p) = \sup \{\Re \lambda: \lambda \in \sigma(\L_p)\}$ its spectral bound. 
\begin{itemize}
	\item  $\big( \L_p, D(\L_p) \big)$ generates an analytic semigroup on $L^p(\Omega)^{n+1}$, see~\cite[Prop.~4.1]{CMCKS01}. 
	\item The spectrum of $\big(\L_p, D(\L_p)\big)$ can be characterized~as
		\begin{nalign}
			\label{thm:SpectrumOfOperatorL}
			\sigma(\L_p) = \sigma\big(\pmb A(\cdot)\big) \cup \lbrace \text{eigenvalues of } \; \L_p \rbrace.
		\end{nalign}
		In particular, the spectrum of $\L_p$ is independent of $p$, see \cite[Thm. 4.5]{CMCKS01}.
\end{itemize}

For a completeness of the exposition, we recall also properties of the spectrum~$\sigma\big( \pmb A(\cdot)\big)$ of the multiplication operator $\pmb A(\cdot)$ acting on $\Lp{}^n$ for $p\in [1, \infty)$ induced by an arbitrary  matrix $\pmb A(x) =\big( a_{i,j}(x) \big)_{i,j=1}^n$ with  $x$-dependent and essentially bounded coefficients. The following results are either well-known or can be found in the works \cite{HabEng,CMCKS01}, \cite[Appendix B]{kowall2021longtime}.

\begin{itemize}
	\item For a square matrix $\pmb A_0$ with constant coefficients, we have \begin{nalign}
		s(\pmb A_0) = \max \{ \Re \lambda : \lambda \text{ is an eigenvalue of } \pmb A_0\}.
	\end{nalign} 
	\item The spectral radius $r(\pmb A_0) = \sup \lbrace |\lambda|: \; \lambda \in \sigma(\pmb A_0) \rbrace$ satisfies
	\begin{nalign}
		\label{eq:SpecRad}
		r(\pmb A_0) \leqslant |A_0| = \left( \sum_{i,j=1}^n a_{i,j}^2 \right)^\frac{1}{2}.
	\end{nalign}
	\item  For more general matrices $\pmb A = \pmb A(x)$, the spectrum has the form (see e.g.~\cite[Lemma~4.3]{CMCKS01})
	\begin{nalign}
		\label{eq:LinearSystemSpectrumAlternativeForm}
		\sigma\big(\pmb A(\cdot) \big) = 
		\left\{ \lambda \in \C: 
		\begin{array}{l}
			\forall{\varepsilon>0} \quad \exists{\Omega_\varepsilon \subseteq \Omega} \quad \text{open set} \quad \exists{ \pmb \xi_\varepsilon \in \R^n} \\ 
			\text{such that} \quad \| \pmb A(\cdot) \pmb \xi_\varepsilon - \lambda \pmb \xi_\varepsilon \|_{L^\infty(\Omega_\varepsilon)} \leqslant \varepsilon |\pmb \xi_\varepsilon|
		\end{array} 
		\right\}.
	\end{nalign} 
 	\item The following equality holds true
	\begin{nalign}
 		\label{SpecBouChar}
 		s\big( \pmb A (\cdot) \big) = \esssup_{x\in \Omega} \big( s(\pmb A(x))\big).
 	\end{nalign}
 	\item \label{4}
 	By \cite[Lemma 4.4]{CMCKS01}, if there exists  a closed set $\Omega'\subseteq \Omega$ such that $|\Omega'|=0$ and \rev{$a_{i,j}\big|_{\Omega\setminus \Omega'}$} is continuous, then
 		\begin{nalign}
 			\label{eq:SpecEq}
 			\sigma\big(\pmb A(\cdot)\big) = \overline{\bigcup_{x\in\Omega\setminus \Omega'} \sigma\big(\pmb A(x)\big)}.
 		\end{nalign} 
	\item  It follows immediately from formulas \eqref{eq:LinearSystemSpectrumAlternativeForm} and \eqref{eq:SpecEq} that eigenvalues of matrix $\pmb A(x)$ belongs to the spectrum of $\L_p$ for almost every $x\in \overline{\Omega}$.
\end{itemize}

Finally, we recall properties of matrices which we use in this work.
\begin{itemize}
	\item \label{MatInv} For an arbitrary matrix $\pmb A_0$ with constant coefficients and  satisfying $s(\pmb A_0) <0$, the resolvent $\left|(\pmb A_0 - \lambda \pmb I)^{-1}\right|$ is uniformly bounded for all $\lambda \in \C$ with $\Re \lambda \geqslant 0$. Indeed, the determinant is a product of all eigenvalues, hence by \eqref{eq:SpecRad} 
	\begin{nalign}
		|\det (\pmb A_0  - \lambda \pmb I)| \geqslant 	\big(\max\left(|s(\pmb A_0)|, |\lambda| - |\pmb A_0|  \right)\big)^n.
	\end{nalign}
	Therefore, denoting by adj the adjugate matrix we obtain 
	\begin{nalign}
		\label{eq:InvEs}
		\left|(\pmb A_0 - \lambda \pmb I)^{-1}\right| &=  \frac{\left| 	{\text{adj}(\pmb A_0 - \lambda\pmb I)} \right|}{|\det (\pmb A_0 - \lambda \pmb I)|} \leqslant \frac{C(\revone{|\lambda|^{n-1}}+1)}{\max\big(|s(\pmb A_0)|, |\lambda| - |\pmb A_0|  \big)^n} \leqslant C.
	\end{nalign}
	Analogously, for every matrix $\pmb A(x)$ with essentially bounded coefficients, equation~ \eqref{SpecBouChar} provides the inequality $s\big( \pmb A(x)\big) \leqslant s\big(\pmb A(\cdot) \big)$ for almost all $x\in \Omega$. Thus, if $s\big( \pmb A(\cdot)\big) < 0$, then 
	\begin{nalign}
		\label{eq:InvEs2}
		\left|(\pmb A(x) - \lambda \pmb I)^{-1}\right| & \leqslant \frac{C(\lambda^n+1)}{\max\big(|s(\pmb A(x))|, |\lambda| - |\pmb A(x)|  \big)^n} \\
		&\leqslant \frac{C(\lambda^n+1)}{\max\big(\left|s\big(\pmb A(\cdot)\big)\right|, |\lambda| - \|\pmb A\|_\Li  \big)^n} \leqslant C,
	\end{nalign}
	for almost all $x\in \Omega$ and all $\lambda \in C$ with $\Re \lambda \geqslant 0$.  
		  
	\item It follows from equation \eqref{SpecBouChar}, that if $s\big(\pmb A(\cdot)\big) <0$ then for every $k\in \big(0, -s(\pmb A(\cdot)\big)$ there exist constants $C>0$ such that
	\label{SemEs}
	\begin{nalign}
		\label{RTG}
		\left| e^{\pmb A(x) t} \right| \leqslant Ce^{-kt} \quad \text{for almost all} \quad x\in \Omega.
	\end{nalign}
	\item \label{DetIden} For matrices $\pmb A, \pmb B, \pmb C, d$ as in  \eqref{eq:LinearSystemCoefficientMatrixDefinition}-\eqref{abcd} and arbitrary $\lambda \in \C$ we have
	\begin{nalign}
		\label{eq:DetIden}
		d - \lambda -  \bC (\bA - \lambda \pmb I) ^{-1} \bB 	&=  \frac{1}{\det \left(\bA - \lambda \pmb I\right)}\det\begin{pmatrix}
		\bA - \lambda \pmb I  & \bB  \\
		\bC  & d- \lambda
		\end{pmatrix}.
	\end{nalign}
	The proof of this identity is given in \cite[eq. (3.5)]{CMCKS01}.
\end{itemize}

\subsection{Conditions for stability} 

We begin by resolvent estimates in the case of constant coefficients. 
\begin{lemma}
	\label{thm:LinearSystemEstimatesForUVByFG}
	Assume that matrices $\pmb A = \pmb A_0$, $\pmb B = \pmb B_0$, $\pmb C = \pmb C_0$, $d = d_0$ in \eqref{eq:LinearSystemCoefficientMatrixDefinition} have constant coefficients and satisfy 
	\begin{nalign}
		\label{eq:LinearSystemStabilityConditionNDimensional}
		s(\pmb A_0)<0 \quad \text{as well as} \quad s\begin{pmatrix} 
		\pmb A_0 & \pmb B_0 \\ \pmb C_0 & d_0
		\end{pmatrix}<0.
	\end{nalign} 
	Let $p\in (1,\infty)$ and $\lambda \in \C$ with $\Re\lambda\geqslant0$. There exists $\gamma_0 = \gamma_0(\pmb A_0, \pmb B_0, \pmb C_0, d_0, |\Omega| ,p)>0$ such that for all $\gamma \geqslant \gamma_0$ the problem
	\begin{nalign}
		\label{eq:LinearSystemEstimatesForUV}
		(\pmb A_0 - \lambda \pmb I) \pmb \varphi + \pmb B_0 \psi &= \pmb f, \quad \text{for } x\in \overline{\Omega}, \\
		\gamma \Delta_\nu \psi + \pmb C_0 \pmb \varphi + (d_0-\lambda)\psi &= g, \quad \text{for } x\in \Omega,
	\end{nalign}
	has a unique solution $(\pmb \varphi, \psi)\in L^p(\Omega)^n\times \W1p$ for each $(\pmb f, g)\in L^p(\Omega)^{n+1}$. Moreover, there exists a number $K=K(\pmb A_0, \pmb B_0, \pmb C_0, d_0, |\Omega|,p, \gamma)>0$ independent of~$\lambda$ and $\pmb f, g$ such that
	\begin{nalign}
	\label{eq:ResLin2}
	\|\pmb \varphi\|_\Lp{} + \|\psi\|_\W1p \leqslant K \big(\|\pmb f\|_\Lp{} + \|g\|_\Lp{}\big).
	\end{nalign}
\end{lemma}

\begin{proof}
	\rev{A unique solution to system \eqref{eq:LinearSystemEstimatesForUV} 
	can be constructed in a standard way, once we prove the {\it a priori} estimate
 \eqref{eq:ResLin2}.} In order to prove inequality \eqref{eq:ResLin2}, we integrate both equations in \eqref{eq:LinearSystemEstimatesForUV} over $\Omega$ to obtain
	\begin{nalign}
		\label{eq:LinearSystemIntegralEstimatesForUV}
		(\pmb A_0 - \lambda \pmb I) \int_\Omega\pmb \varphi\dx + \pmb B_0 \int_\Omega \psi\dx &= \int_\Omega\pmb f\dx, \\
		\pmb C_0 \int_\Omega\pmb \varphi\dx + (d_0-\lambda)\int_\Omega \psi\dx &= \int_\Omega g\dx.
	\end{nalign}
	Since $s\begin{spm}
	\pmb A_0 &\pmb B_0 \\ \pmb C_0 & d_0 \end{spm} < 0$, applying inequality \eqref{eq:InvEs}, we have
	\begin{nalign}
		\label{eq:LinearSystemUVMassEstiamte}
		\left| \int_\Omega \pmb \varphi \dx \right| + \left| \int_\Omega \psi \dx \right| &\leqslant C 	\left(\left|  \int_\Omega \pmb f	 \dx \right| + \left| \int_\Omega g\dx \right| \right) \\ & \leqslant C |\Omega|^\frac{p-1}{p}  \big(\|\pmb f\|_\Lp{} + \|g\|_\Lp{}\big)
	\end{nalign}
	with a constant $C>0$ independent of $\lambda$. Subtracting equations \eqref{eq:LinearSystemIntegralEstimatesForUV} multiplied by $\frac{1}{|\Omega|}$ from equations~\eqref{eq:LinearSystemEstimatesForUV} and denoting the \rev{deviations from the} averages 
	\begin{nalign}
		\tilde{\pmb \varphi} &= \pmb \varphi - \frac{1}{|\Omega|} \int_\Omega \pmb \varphi \dx, & 
		\tilde{\psi} &= \psi - \frac{1}{|\Omega|} \int_\Omega \psi \dx, \\	
		\tilde{\pmb f} &= \pmb f - \frac{1}{|\Omega|} \int_\Omega \pmb f\dx, &
		\tilde{g} &= g - \frac{1}{|\Omega|} \int_\Omega g \dx, \\		
	\end{nalign}
	we obtain
	\begin{nalign}
		\label{eq:LinearSystemEstimatesForUVAverages}
		(\pmb A_0 - \lambda \pmb I) \tilde{\pmb \varphi} + \pmb B_0 \tilde{\psi} &= \tilde{\pmb f}, \quad \text{for } x\in	 \overline{\Omega}, \\
		\gamma \Delta_\nu \tilde{\psi} + \pmb C_0 \tilde{\pmb \varphi} + (d_0-\lambda)\tilde{\psi} &= \tilde{g}, \quad \text{for } x\in \Omega.
	\end{nalign}
	Thus, by first equation in \eqref{eq:LinearSystemEstimatesForUVAverages} and by inequality \eqref{eq:InvEs}, we have
	\begin{nalign}
		\label{eq:LinearSystemPointwiseUEstimate}
		| \tilde{\pmb \varphi}(x)|&\rev{  =\left| (\pmb A_0 -\lambda \pmb I)^{-1} \left(\tilde{\pmb f}(x) - \pmb B_0 \tilde\psi(x) \right)\right|} \leqslant C \big( |\tilde{\pmb f}(x)| + |\tilde{\psi}(x)| \big)
	\end{nalign}
	for all $x\in \overline{\Omega}$.
	Next, applying the elliptic regularity for $\Delta_\nu$, (see eg.~\cite[Thm.~3.1.3]{MR3012216}) to the second equation in \eqref{eq:LinearSystemEstimatesForUVAverages} we obtain the estimate
	\begin{nalign}
		\label{eq:GradP}
		\int_\Omega |\nabla \tilde{\psi}|^p \dx \leqslant \frac{C}{\gamma} \left(\int_\Omega |\tilde{\psi}|^p \dx + \int_\Omega|\tilde{\pmb	 f}|^p\dx + \int_\Omega|\tilde{g}|^p\dx\right).
	\end{nalign}
	Next, applying the Poincar\'e inequality with a constant $C = C(|\Omega|,p)$ 
	\begin{nalign}
		\int_\Omega|\nabla \tilde{\psi}|^p \dx \geqslant C \int_\Omega|\tilde{\psi}|^p \dx
	\end{nalign}
	for each $\gamma$ large enough, we obtain
	\begin{nalign}
		\int_\Omega |\nabla\tilde{\psi}|^p \dx \leqslant C \left( \int_\Omega|\tilde{\pmb f}|^p\dx + \int_\Omega|\tilde{g}|^p \dx	 \right).
	\end{nalign}
	Consequently using \eqref{eq:LinearSystemUVMassEstiamte} and \eqref{eq:GradP} we estimate
	\begin{nalign}\rev{
		\|\psi\|_\W1p \leqslant \|\tilde{\psi}\|_\W1p + |\Omega|^\frac{1}{p}\left|\frac{1}{|\Omega|}\int_\Omega \psi \dx \right| \leqslant	 K\Big(\|\pmb f\|_\Lp{} + \|g\|_\Lp{} \big).}
	\end{nalign}
	Finally, we get the required estimate for $\pmb \varphi$ from inequalities \eqref{eq:LinearSystemPointwiseUEstimate} and \eqref{eq:LinearSystemUVMassEstiamte}.
\end{proof}

\begin{theorem}
	\label{thm:ResEsNonCon}	
	Assume that constant coefficient matrices $\pmb A_0$, $\pmb B_0$, $\pmb C_0$, $d_0$ satisfy
	\begin{nalign}
		\label{eq:ThmNeq}
		s(\pmb A_0)<0, \quad s\begin{pmatrix} 
			\pmb A_0 & \pmb B_0 \\ \pmb C_0 & d_0
		\end{pmatrix}<0
	\end{nalign} 	
	and fix $\gamma>0$ so large as required in Lemma \ref{thm:LinearSystemEstimatesForUVByFG}. Let $\big(\mathcal{L}_p, D(\mathcal{L}_p)\big)$ be the operator defined in \eqref{eq:LinearSystemOperatorLDefinition} with matrices $\pmb A(\cdot)$, $\pmb B(\cdot)$, $\pmb C(\cdot)$, $d(\cdot)$  as in \eqref{eq:LinearSystemCoefficientMatrixDefinition}-\eqref{abcd}.
 	Assume, moreover, that 
 	\begin{nalign}
 		\label{eq:Acon}
 		s(\pmb A(\cdot))<0. 
 	\end{nalign}
 	There \rev{exists} $\rho>0$ such that if
	\begin{nalign}
		\label{MatRho}
		\|\pmb A - \pmb A_0\|_\Lp{N} +
		\|\pmb B - \pmb B_0\|_\Lp{N} +
		\|\pmb C - \pmb C_0\|_\Lp{N} +
		\|d - d_0\|_\Lp{N} \leqslant\rho
	\end{nalign}	 
	then $s\left(\L_p\right)<0$.  
\end{theorem}

\begin{proof}
	By equation \eqref{thm:SpectrumOfOperatorL}, it suffices to show that for all $\lambda \in \C$ with $\Re \lambda \geqslant 0$, the zero solution is the only solution to the problem
	\begin{nalign}
		\label{eq:LinearEigenvalueEquation2N}
		\pmb A  \pmb \varphi + \pmb B \psi &= \lambda \pmb \varphi, && \text{for } x\in \overline{\Omega}, && \\
		\gamma \Delta_\nu \psi + \pmb C \pmb \varphi + d \psi  &= \lambda \psi, && \text{for } x\in \Omega.  
	\end{nalign}
	Notice that for each $\lambda \in \C$ with $\Re \lambda \geqslant 0$, by inequality \eqref{eq:InvEs2} the function
	\begin{nalign}
	\pmb \varphi(x) = - (\pmb A(x) -\lambda \pmb I)^{-1} \pmb B(x) \psi(x)
	\end{nalign}
	satisfies
	\begin{align}
		\label{eq:LinearEigenvalueEquationUByVEstimate}
		| {\pmb \varphi}(x)|  \leqslant \big|(\pmb A -\lambda \pmb I)^{-1} \pmb B \psi(x) \big|   \leqslant  C  |\psi(x)|,
	\end{align}
	for all $x\in \Omega$. 
	Now, we rewrite problem \eqref{eq:LinearEigenvalueEquation2N} in the following form
	\begin{nalign}
		\label{eq:LinearEigenvalueEquationSeparate2N}
		(\pmb A_0 -\lambda \pmb I) \pmb \varphi + \pmb B_0 \psi &= - (\pmb A - \pmb A_0) \pmb \varphi  - (\pmb B -\pmb B_0) \psi, && \text{for } x\in \overline{\Omega}, 	&&\\
		\gamma \Delta_\nu \psi + \pmb C_0 \pmb \varphi + (d_0 - \lambda) \psi &= - (\pmb C - \pmb C_0) \pmb \varphi - (d - d_0) \psi, && \text{for } x\in \Omega. &&
	\end{nalign}
	By Lemma \ref{thm:LinearSystemEstimatesForUVByFG}, there exists a constant $C >0$ such that the solution of equation \eqref{eq:LinearEigenvalueEquationSeparate2N} satisfies 
	\begin{nalign}
		\|\pmb \varphi\|_\Lp{} + \|\psi\|_\W1p & \leqslant C \left( \|(\pmb A - \pmb A_0)\pmb \varphi\|_\Lp{} + \|(\pmb B -\pmb B_0)  \psi\|_\Lp{}\right) \\ & + C \left( \| (\pmb C - \pmb C_0) \pmb \varphi\|_\Lp{}	+  \| (d - d_0) \psi\|_\Lp{}\right).
	\end{nalign}
	For \rev{$N> p$}, by the H\"older and Sobolev inequalities, we estimate
	\begin{nalign}
		\| (\pmb B -\pmb B_0) \psi\|_\Lp{} &+ \| (d - d_0) \psi\|_\Lp{} \\ &\leqslant  \|\psi\|_\Lp{\frac{Np}{N-p}} \left( \| (\pmb B -\pmb B_0)\|_\Lp{N} + \|(d - d_0)\|_\Lp{N} \right) 	\\ &\leqslant  C \|\psi\|_\W1p \left(\| (\pmb B -\pmb B_0)\|_\Lp{N} + \|(d - d_0)\|_\Lp{N}  \right)
	\end{nalign}
	and similarly, by the inequality \eqref{eq:LinearEigenvalueEquationUByVEstimate} combined with the Sobolev inequality, we obtain
	\begin{nalign}
		\|(\pmb A - \pmb A_0) \pmb \varphi\|_\Lp{} &+ \| (\pmb C - \pmb C_0) \pmb \varphi\|_\Lp{}\\ &\leqslant  \|\pmb \varphi\|_\Lp{\frac{Np}{N-p}} \big( \|(\pmb A - \pmb A_0) \|_\Lp{N} 	+ \| (\pmb C - \pmb C_0)\|_\Lp{N} \big) \\ &\leqslant  C \|v\|_\W1p \left(\|(\pmb A - \pmb A_0) \|_\Lp{N} + \|\pmb (\pmb C - \pmb C_0)\|_\Lp{N} \right).
	\end{nalign}  
	Choosing $\|(\pmb A - \pmb A_0)\|_\Lp{N}+\|(\pmb B - \pmb B_0)\|_\Lp{N}+\|\pmb (\pmb C - \pmb C_0)\|_\Lp{N} +\|d - d_0\|_\Lp{N}$ sufficiently small, we conclude that $\pmb \varphi =0$ and $\psi = 0$.
	 
	\rev{By a similar reasoning, an analogous inequality can be obtained for $p\geq N$ using the embedding $W^{1,p}(\Omega)\subset L^q(\Omega)$ for each $q\geq p$.
	The proof of the inequality  $s(\L_p)<0$  for each $p\in(1,\infty)$ 
	may be also completed by using the fact that 
	 the spectrum of the operator $\L_p$ is independent of $p$, see inequality \eqref{thm:SpectrumOfOperatorL} and the comment following it.}
\end{proof}

In the particular case of one ODE coupled with one reaction-diffusion equation, we can formulate sufficient conditions for stability for every diffusion coefficient $\gamma>0$. 

\begin{prop}[One PDE coupled with one ODE]
	\label{thm:LinearStabilitystantNoncoefficientsN1}
	Let $\big( \L_p, \, D(\L_p) \big)$ be the operator given by formula \eqref{eq:LinearSystemOperatorLDefinition}. Assume that $n=1$, thus $\pmb A(x)= a(x)$, $\pmb B(x)= b(x)$, $\pmb C(x)= c(x)$, $d(x)$ are essentially bounded functions.  If
	\begin{nalign}
		s\big( a(\cdot)\big)<0 , \quad s\begin{spm} 
			a(\cdot) & b(\cdot) \\ c(\cdot) & d(\cdot)
		\end{spm}<0, \quad s\big(d(\cdot)\big)<0
	\end{nalign}  
	then $s(\L_p) <0$. 
\end{prop}

\begin{proof}
	Since $s(a(\cdot))<0$, by equation \eqref{thm:SpectrumOfOperatorL},  it suffices to study only eigenvalues of the operator $\L_p$. We prove that $(\varphi,\psi) = (0,0)$ is the only solution of the system
	\begin{nalign}
		\label{eq:LinearEigenvalueEquation1N}
		a \varphi + b \psi &= \lambda  \varphi, 					   && x\in \overline{\Omega}, \\
		\gamma \Delta_\nu \psi +  c\varphi + d\psi &= \lambda \psi, && x\in {\Omega}
	\end{nalign}
	for all $\lambda \in \C$ with $\Re \lambda \geqslant -\kappa$ for some
	\begin{nalign} 
		\kappa \in \bigg(0,\ \min \left(-s\big(a(\cdot)\big),-s\begin{spm} 
			a(\cdot) & b(\cdot) \\ c(\cdot)& d(\cdot)
		\end{spm}, -s\big(d(\cdot)\big) \right)\bigg).
	\end{nalign} 
	Since $a(x)\leqslant s(a(\cdot)) <0$ almost everywhere, first equation in \eqref{eq:LinearEigenvalueEquation1N} implies  
	\begin{align}
		\label{eq:LinearEigenvalueUInverse1N}
		\varphi = -\frac{b}{a-\lambda} \psi,\quad \text{for} \quad x\in \Omega. 
	\end{align} 
	Substituting this expression into second equation in \eqref{eq:LinearEigenvalueEquation1N} we obtain
	\begin{align}
		\label{eq:LinearEigenvalueVEquation}
		\gamma\Delta_\nu \psi + \Big(-\frac{cb}{a - \lambda } + d - \lambda\Big) \psi = 0, \quad \text{for} \quad x\in \Omega.
	\end{align}	
	We introduce the function
	\begin{nalign}
		\label{eq:LinearSystemHLambdaFunctionDefinition}
		h(x,\lambda) &\equiv \Re \Big(-\frac{c(x)b(x)}{a(x) - \lambda } + d(x) - \lambda\Big) \\
		&= -\frac{c(x)b(x)a(x) - c(x)b(x)\Re \lambda}{(a(x)-\Re \lambda)^2 + \operatorname{Im}^2\lambda} + d(x) - \Re \lambda,
	\end{nalign}
	which, by equation \eqref{eq:DetIden}, is also given by the formula
		\begin{nalign}
		\label{eq:HDetDef}
		h(x,\lambda) &= \Re \left(\frac{1}{a(x)-\lambda}{\det \begin{pmatrix}
				a(x) -\lambda & b(x) \\ c(x) & d(x) - \lambda
		\end{pmatrix}}\right).
	\end{nalign}
	It is well-known that zero is the only solution of equation \eqref{eq:LinearEigenvalueVEquation} if  	
	\begin{nalign}
		\esssup_{x\in \Omega, \; \lambda \in \C, \ \Re \lambda \geqslant -\kappa} h(x,\lambda) <0.
	\end{nalign} 
	\rev{In order to prove that this inequality is satisfied, first we consider $\lambda = \alpha \geqslant -\kappa$.} It follows from assumption on $\alpha$ that
	\begin{nalign}
		s\begin{pmatrix} 
			a(x) - \alpha & b(x) \\ c(x)& d(x) - \alpha
		\end{pmatrix}<0
	\end{nalign}
	implies
	\begin{nalign} 
		\det \begin{pmatrix} 
			a(x) - \alpha & b(x) \\ c(x) & d(x) - \alpha
		\end{pmatrix} \geqslant s^2\begin{pmatrix} 
			a(x) - \alpha & b(x) \\ c(x)& d(x) - \alpha
		\end{pmatrix} >0 
	\end{nalign}
	for almost all $x\in \Omega$. Thus, the inequalities $a(x)\leqslant s(a(\cdot))<-\alpha<0$ \rev{yields} 
	\begin{nalign}
		\label{eq:LinearSytemHfunctionEstimation}
		h(x,\alpha) &= \frac{\det \begin{pmatrix}
		a(x) -\alpha & b(x) \\ c(x) & d(x) - \alpha
		\end{pmatrix}}{a(x)-\alpha}
		&\leqslant -\frac{\left(s\begin{pmatrix} 
		a(\cdot) & b(\cdot) \\ c(\cdot)& d(\cdot)
		\end{pmatrix} - \alpha\right)^2}{ |s(a)| + |\alpha|}. 
	\end{nalign} 
		Since $s\begin{spm} a(\cdot) & b(\cdot) \\ c(\cdot)& d(\cdot)
		\end{spm} < 0$ it is clear that there exist $h_0>0$ such that 
		\begin{nalign}
			\rev{-\frac{\left(s\begin{pmatrix} 
			a(\cdot) & b(\cdot) \\ c(\cdot)& d(\cdot)
			\end{pmatrix} - \alpha\right)^2}{ \|a\|_\infty + |\alpha|}} < 	-h_0 < 0 \quad \text{for all} \quad \alpha \geqslant -\kappa.
		\end{nalign}
	
		Finally, by formula \eqref{eq:LinearSystemHLambdaFunctionDefinition}, the function $h(x,\alpha + i\beta)$ with fixed $\alpha\geqslant-\kappa$ is even with respect to $\beta\in \R$  and monotone with respect to $\beta^2$. Thus, for the number $h_0 >0$ defined above and by formula \eqref{eq:LinearSystemHLambdaFunctionDefinition} we have
		\begin{nalign}
			h(x,\lambda) &\leqslant \max(s(d(\cdot)), -h_0) < 0,
		\end{nalign}
		for all $\lambda = \alpha + i \beta$ with $\alpha \geqslant -\kappa$ and $\beta\in \R$ and almost all $x\in \Omega$. 
\end{proof}

\begin{rem}
	Proposition \ref{thm:LinearStabilitystantNoncoefficientsN1} can not be directly extended to systems with more than one ODE because, \rev{in this case, 
	the counterpart of 
	the function $h(x,\alpha + i\beta)$ defined in \eqref{eq:LinearSystemHLambdaFunctionDefinition}}
	 is no longer monotone with respect to $\beta>0$. 
\end{rem}

\section{Nonlinear stability} 

\label{sec:Nonlinear Stability}
Using stability results for the linear system from the previous section, we are in a position to  study a stability of the zero solution to the nonlinear problem
\begin{nalign}
	\label{eq:NonlinearStabilityLinearisedProblemWithRest}
	\frac{\partial}{\partial t}
	\begin{pmatrix}
		\pmb \varphi \\ \psi
	\end{pmatrix} 
	= \begin{pmatrix}
		0 \\ \gamma\Delta_\nu \psi
	\end{pmatrix} + 
	\begin{pmatrix}
		\pmb A & \pmb B \\ \pmb C & d
	\end{pmatrix} 
	\begin{pmatrix}
		\pmb \varphi \\ \psi
	\end{pmatrix} 
	+ 
	\begin{pmatrix}
		\pmb R_1  (\pmb \varphi, \psi) \\ R_2(\pmb \varphi, \psi)
	\end{pmatrix}
	&& \text{for} \quad x\in {\Omega}, \quad t>0,
\end{nalign}
with arbitrary $x$-dependent  matrices  $\pmb A$,   $\pmb B$,   $\pmb C$,   and $d$ of the form 
\eqref{eq:LinearSystemCoefficientMatrixDefinition}--\eqref{abcd} and with the remainders
\begin{nalign}
	\pmb R_1 = \begin{pmatrix}
		R_{11}(\pmb \varphi, \psi) \\ \vdots \\ R_{1n}(\pmb \varphi, \psi)
	\end{pmatrix} \quad \text{and} \quad R_2 = R_2(\pmb \varphi, \psi),
\end{nalign}
satisfying the following typical assumption.

\begin{assumption}
	\label{ass:SquareEs}
	There exist constants $K>0$ and $C_K=C(K)>0$ such that for all $ {\pmb \varphi}\in \R^n$ and ${\psi} \in \R$ satisfying $|\pmb \varphi| + |\psi| \leqslant K$ we have 
	\begin{nalign}
		\retainlabel{RemEs}
		|\pmb R_1(\pmb \varphi, \psi)| \leqslant C_K \big( |\pmb \varphi|^2 + |\psi|^2 \big) \quad \text{and}\quad | R_2(\pmb \varphi, \psi)| \leqslant C_K \big( |\pmb \varphi|^2 + |\psi|^2 \big). 
	\end{nalign} 
\end{assumption}

We supplement system~\eqref{eq:NonlinearStabilityLinearisedProblemWithRest} with an arbitrary initial condition 
\begin{nalign} 
	\label{eq:IniNonLin}
	\pmb \varphi(\cdot,0) =\pmb \varphi_0 \in \Li,\quad \psi(\cdot,0) = \psi_0 \in C(\overline{\Omega}),
\end{nalign} 
and we assume that initial value problem \eqref{eq:NonlinearStabilityLinearisedProblemWithRest}-\eqref{eq:IniNonLin} has a unique global-in-time mild solution 
\begin{nalign}
	\label{eq:GitSol}
	\pmb \varphi \in C\big([0,\,\infty),\, \Li^n \big)\quad \text{and} \quad \psi\in C\big([0,\,\infty), \, C(\oo)\big),
\end{nalign}
namely, the functions $(\pmb \varphi,\psi)$ satisfying equations \eqref{eq:PointEsPhi} and \eqref{eq:PsiLiEs}, below. Such a unique local-in-time solution can be constructed following the reasoning from the Theorem~\ref{thm:ExPar} below and it can be extended to a global-in-time solution using \textit{a priori} estimates below. 

\begin{theorem}[Stability of zero solution]
	\label{thm:NonlinearStability}
	Assume that
	\begin{itemize}
		\item $s(\L_p)<0$, where $\L_p$ is defined by \eqref{eq:LinearSystemOperatorLDefinition},
		\item $s\big(d(\cdot)\big) <0$,		
		\item Assumption \ref{ass:SquareEs} holds true for some $K>0$.
	\end{itemize}
	Then the zero solution to system \eqref{eq:NonlinearStabilityLinearisedProblemWithRest}-\eqref{eq:IniNonLin} is exponentially asymptotically stable in $\Li^{n+1}$. More precisely, there exist $C\geq 1$  and $k>0$ such that if
	\begin{nalign}
		\|\pmb \varphi_0\|_\Li + \|\psi_0\|_\Li <\frac{K}{4C^2},
	\end{nalign} 	
	then
	\begin{nalign}
		\label{eq:NonDec}
		\|\pmb \varphi(t)\|_\Li + \|\psi(t)\|_\Li \leqslant C \big(\|\pmb \varphi_0\|_\Li + \|\psi_0\|_\Li\big) e^{-\frac{k}{2} t},
	\end{nalign}
	for all $t\geq 0$.
\end{theorem}
\rev{Let us briefly sketch the proof of Theorem \ref{thm:NonlinearStability}. First, we use the assumption~$s(\L_p)<0$ to show the $L^p$-$L^\infty$ estimates for the couple $(\pmb \varphi, \psi)$ in Lemma \ref{thm:NonlinearStabilityDuhamel}. 
Then, using the assumption $s\big(\pmb A(\cdot)\big)<0$ (which is a direct  consequence of the inequality $s(\L_p)<0$, see equation \eqref{thm:SpectrumOfOperatorL}) we show the pointwise estimate for the function $\pmb \varphi$ in Lemma \ref{thm:NonlinearStabilityFirstPointwiseEstimate}. 
In Lemma \ref{thm:NonlinearStabilityPsiByLinAndSquareEstimate} we use the assumption $s\big(d(\cdot)\big)<0$, to obtain the $L^\infty$-$L^p$ estimate for the function $\psi$, which together with $L^p$-$L^{\infty}$ estimates for $\pmb \varphi$ from Lemma \ref{thm:NonlinearStabilityOneToSquareEstimate} provides the desired $L^\infty$-$L^\infty$ estimates in Lemma \ref{thm:NonlinearStabilityPsiOmegaTylda}. 
}

\begin{lemma}
	\label{thm:NonlinearStabilityDuhamel}
	\rev{Assume that $s(\L_p)<0$ and let} $(\pmb \varphi, \psi)$ be a solution \eqref{eq:GitSol} to system \eqref{eq:NonlinearStabilityLinearisedProblemWithRest}-\eqref{eq:IniNonLin}  satisfying
	\begin{nalign}
		\sup_{t\in [0,\,T]} \|\pmb \varphi(t)\|_\Li +\revone{\sup_{t\in [0,\,T]}} \|\psi(t)\|_\Li \leqslant K \quad \text{for some} \ \ T>0.
	\end{nalign}
	For each $p\in (1,\infty)$ and each $k_1\in \big(0, \, -s(\L_p)\big)$ there exists $C=C(K, k_1, p) > 0$ such that
	\begin{nalign}
		\|\pmb \varphi(t) \|_\Lp{} + \|\psi (t)\|_\Lp{} &\leqslant C e^{-k_1 t} \big(\|\pmb \varphi_0 \|_\Li + \|\psi_0 \|_\Li \big) \\ &\rev{+ C \int_0^t e^{-k_1(t-s)} \big(\| \pmb \varphi(s)\|_\Li^2 + \|\psi(s)\|_\Li^2\big) \ds},
	\end{nalign}
	for all $t\in [0, \, T]$.
\end{lemma}

\begin{proof} 
	We begin by the integral formulation of the Cauchy problem \eqref{eq:NonlinearStabilityLinearisedProblemWithRest}-\eqref{eq:IniNonLin} 
	\begin{nalign}
		\label{eq:NonlinearStabilityDuhamelPrinciple}
		\begin{pmatrix}	
			\pmb \varphi(t) \\ \psi(t) 
		\end{pmatrix}
		= e^{t \L_p} \begin{pmatrix}\pmb \varphi_0 \\ \psi_0 \end{pmatrix} + \int_0^t e^{(t-s)\L_p}  	
		\begin{pmatrix}
			\pmb R_1  \big(\pmb \varphi(s), \psi(s) \big) \\ R_2\big(\pmb \varphi(s), \psi(s)\big)
		\end{pmatrix} 
		\ds.
	\end{nalign} 
	Since $s(\L_p) <0$, it is well-known (see \textit{e.g.}~\cite[Ch.~IV,~Cor.~3.12]{MR1721989} that for each $k_1\in\big(0,\, -s(\L_p)\big)$  \rev{there exists a constant} $C = C(k_1)>0$ such that
	\begin{nalign}
		\left\| e^{t\L_p} \begin{pmatrix}	\pmb \varphi_0 \\ \psi_0 \end{pmatrix} \right\|_\Lp{} \leqslant C e^{-k_1t} 	\left\| \begin{pmatrix}	\pmb \varphi_0 \\ \psi_0 \end{pmatrix} \right\|_\Lp{}  \text{ for all } t\geqslant 0. 
	\end{nalign}
	Now, it suffices to compute the $L^p$-norm of equation \eqref{eq:NonlinearStabilityDuhamelPrinciple}, apply the well known inequality $\| w\|_\Lp{} \leqslant |\Omega|^\frac{1}{p}\| w\|_\Li$ and use Assumption \ref{ass:SquareEs} for the remainders. 
\end{proof}

\begin{lemma}
	\label{thm:NonlinearStabilityFirstPointwiseEstimate}
	\rev{Assume that $s(\L_p)<0$ and let} $(\pmb \varphi, \psi)$ be a solution \eqref{eq:GitSol} to system \eqref{eq:NonlinearStabilityLinearisedProblemWithRest}-\eqref{eq:IniNonLin} satisfying
	\begin{nalign}
		\sup_{t\in [0,\,T]} \|\pmb \varphi(t)\|_\Li +\revone{\sup_{t\in [0,\,T]}} \|\psi(t)\|_\Li \leqslant K \quad \text{for some} \ \ T>0.
	\end{nalign} 
	For each $k_2\in \big(0, \, -s(\pmb A(\cdot))\big)$ there exist a number $C = C(K,k_2) >0$ such that
	\begin{nalign}
		|\pmb \varphi(x,t)| &\leqslant C e^{-k_2 t} |\pmb \varphi_0(x)| + C \int_0^t e^{-k_2(t-s)} \big(| \psi(x,s)| + |\psi(x,s)|^2 + |\pmb \varphi (x,s)|^2 \big)\ds,
	\end{nalign}
	for all $t\in [0, \, T]$ and almost all $x\in \Omega$.
\end{lemma}

\begin{proof}
	We write the equation for $\pmb \varphi$ in system \eqref{eq:NonlinearStabilityLinearisedProblemWithRest} in the following way
	\begin{nalign}
		\label{eq:PointEsPhi}
		\pmb \varphi(x,t) = e^{t\pmb A(x)} \pmb \varphi_0(x) + \int_0^t e^{(t-s)\pmb A(x)} \big( \pmb B(x) \psi(x,s) + \pmb R_1(\pmb \varphi(x,s), \psi(x,s)) \big)\ds.
	\end{nalign}
	By equation \eqref{thm:SpectrumOfOperatorL}, the assumption $s(\L_p)<0$ implies $s\big(\pmb A(\cdot) \big)<0$. Therefore, by inequality \eqref{RTG} we have the exponential pointwise decay estimate
	\begin{nalign}
		\label{eq:PointDec}
		\left| e^{t\pmb A(x)} \pmb \varphi_0(x) \right| \leqslant C_1 e^{-k_2 t}|\pmb \varphi_0(x)| \quad \text{for all } t>0 \text{ and almost all } x\in \Omega.
	\end{nalign}
	To complete the proof, it suffices to apply inequality \eqref{eq:PointDec} to equation \eqref{eq:PointEsPhi} and use $\pmb B \in L^\infty(\Omega)^n$ together with Assumption \ref{ass:SquareEs} for the remainder.
\end{proof}

\begin{lemma}
	\label{thm:NonlinearStabilityPsiByLinAndSquareEstimate}
	\rev{Assume that $s\big(d(\cdot)\big)<0$} and let $(\pmb \varphi, \psi)$ be a solution \eqref{eq:GitSol} to system \eqref{eq:NonlinearStabilityLinearisedProblemWithRest}-\eqref{eq:IniNonLin} satisfying
	\begin{nalign}
		\sup_{t\in [0,\,T]} \|\pmb \varphi(t)\|_\Li +\revone{\sup_{t\in [0,\,T]}} \|\psi(t)\|_\Li \leqslant K \quad \text{for some} \ \ T>0.
	\end{nalign}
 	For each $k_3 \in \big(0,\, -s(d(\cdot))\big)$ and each $p \in( N/2, \, \infty)$ \rev{there exists} $C = C(K,k_3, p) > 0$ such that
	\begin{nalign}
		\| \psi(t) \|_\Li &\leqslant Ce^{-k_3 t} \| \psi_0 \|_\Li  \\ & \rev{+ C\int_0^t e^{-(t-s)k_3} \Big(m_p(t-s) \| \pmb \varphi(s)\|_\Lp{} + \|\pmb \varphi(s)\|^2_\Li + \|\psi(s)\|^2_\Li \Big) \ds }
	\end{nalign}
	for all $t\in [0, \, T]$, where $m_p(t) = \max\left(t^{-\frac{N}{2p}},1\right)$.
\end{lemma}

\begin{proof}
	We write the equation for $\psi$ from system \eqref{eq:NonlinearStabilityLinearisedProblemWithRest} in the form
	\begin{nalign}
		\label{eq:PsiLiEs}
		\psi(t) = e^{t(\gamma\Delta_\nu + d)} \psi_0 + \int_0^t e^{(t-s)(\Delta_\nu +d)}\big( \pmb C \pmb\varphi +  R_2(\pmb \varphi, \psi)\big) \ds.
	\end{nalign} 
	Next, we recall the estimate for the Neumann heat semigroup (see \textit{e.g.}~\cite[p.~25]{MR755878}) which holds true for all $\zeta \in \Lp{}$:
	\begin{nalign}
		\left\| e^{t\gamma\Delta_\nu}  \zeta \right\|_\Li \leqslant C m_p(t) \|\zeta\|_\Lp{} \quad \text{for all} \quad t>0 \quad \text{and} \quad p\in [1,\infty].
	\end{nalign} 
	Combining this fact with the maximum principle we conclude that
	\begin{nalign}
		\left\| e^{t(\gamma\Delta_\nu+d)}  \zeta \right\|_\Li \leqslant C m_p(t)e^{-k_3t} \|\zeta\|_\Lp{} \quad \text{for all} \quad t>0 \quad \text{and} \quad p\in [1,\infty]
	\end{nalign}
	 for each $k_3 \in \big(0, \, -s(d(\cdot))\big)$.  
	 \rev{Notice that the function $m_p(t)$ reduces to a constant for $p=\infty$.}
	 Therefore, there exists a constant $C>0$ such that
	\begin{nalign}
		\rev{\|e^{t(\gamma\Delta_\nu +d)}\left( \pmb \varphi + \pmb \varphi^2+ \psi^2 \right)\|_\Li} &\rev{\leqslant\|e^{t(\gamma\Delta_\nu +d)} \pmb \varphi \|_\Li + \|e^{t(\gamma\Delta_\nu +d)}\left(\pmb \varphi^2+ \psi^2 \right)\|_\Li} \\
		&\rev{\leqslant Ce^{-k_3 t}  \Big( m_p(t)\| \pmb \varphi \|_\Lp{} +  \|\pmb \varphi\|_\Li^2+\|\psi\|_\Li^2\Big).}
	\end{nalign}
	We complete this reasoning by calculating the $L^\infty$-norm of equation \eqref{eq:PsiLiEs} and using the assumption $\pmb C \in L^\infty(\Omega)^n$ as well as the Assumption \ref{ass:SquareEs} for the remainder.
\end{proof}

\begin{lemma}
	\label{thm:NonlinearStabilityOneToSquareEstimate}
	\rev{Assume that $s(\L_p)<0$ and $s\big(d(\cdot)\big)<0$.} Let $(\pmb \varphi, \psi)$ be a solution \eqref{eq:GitSol} to system \eqref{eq:NonlinearStabilityLinearisedProblemWithRest}-\eqref{eq:IniNonLin} satisfying
	\begin{nalign}
		\sup_{t\in [0,\,T]} \|\pmb \varphi(t)\|_\Li +\revone{\sup_{t\in [0,\,T]}} \|\psi(t)\|_\Li \leqslant K \quad \text{for some} \ \ T>0.
	\end{nalign}
	Choose the exponents $k_1 \in \big(0, \, -s(\L_p)\big)$ and $k_3 \in \big(0, \, -s(d)\big)$ from Lemmas \ref{thm:NonlinearStabilityDuhamel} and~\ref{thm:NonlinearStabilityPsiByLinAndSquareEstimate} respectively.
	For each $k_4 \in \big(0, \, \min(k_1,k_3)\big)$ and every $p > N/2$ there \revone{exists} $C = C(K,k_4, p) > 0$ such that
	\begin{nalign}
		\int_0^t e^{-k_3(t-s)} m_p(t-s)\|\pmb \varphi(s) \|_\Lp{} \ds &\leqslant C \cdot e^{-k_4 t} 
		\big( \| \pmb \varphi_0 \|_\Li + \| \psi_0 \|_\Li\big) \\ & + C \int_0^t e^{-k_4(t-s)} \big( \| \pmb \varphi(s) \|_\Li^2 + \| \psi(s) \|_\Li^2 \big) \ds
	\end{nalign}
	for all $t\in [0, \, T]$, where $m_p(t) = \max\left(t^{-\frac{N}{2p}},1\right)$.
\end{lemma}

\begin{proof}
	We apply the estimate of $\| \pmb \varphi(t) \|_\Lp{}$ from Lemma \ref{thm:NonlinearStabilityDuhamel} to obtain
	\begin{nalign}
	\int_0^t e^{-k_3(t-s)} m_p(t-s)\|\pmb \varphi(s) \|_\Lp{p} ds \leqslant I_1 + I_2,
	\end{nalign}
	where
	\begin{nalign}  
	I_1 &= C\big(\|\pmb \varphi_0 \|_\Li + \|\psi_0 \|_\Li \big)\int_0^t e^{-k_3(t-s)} m_p(t-s) e^{-k_1 s} \ds, \\ 
	I_2 &=  C \int_0^t e^{-k_3(t-s)} m_p(t-s)\int_0^s e^{-k_1(s-\tau)} \big(\| \pmb \varphi(\tau)\|_\Li^2 + \|\psi(\tau)\|_\Li^2\big)\dtau \ds.
	\end{nalign}
	\rev{Since $m_p(t-s)$ is integrable for $s\in [0,\, t]$ and $p>\frac{N}{2}$,} we estimate the integral in $I_1$ by 
	\begin{nalign}
		\label{I1Es}
		 & e^{-k_3t}\int_0^{t-1} m_p(t-s) e^{(k_3-k_1) s} \ds + C e^{-k_3t}\int_{t-1}^t m_p(t-s) e^{(k_3-k_1) s} \ds  \\  
		&\leqslant  \rev{C e^{-k_3t}\int_0^{t-1}  \max\left(1,e^{(k_3-k_1)s}\right) \ds + Ce^{-k_3t}\int_{t-1}^t m_p(t-s) e^{(k_3-k_1) s} \ds} \\ 
		&\leqslant \rev{C \left(e^{-k_1 t} +  e^{-k_3 t} \right) \int_0^{t-1}  \ds +  C\left(e^{-k_1 t} +  e^{-k_3 t} \right)} \\
		& \leqslant  C(t+1)\left(e^{-k_1 t} +  e^{-k_3 t} \right).
	\end{nalign}
	Therefore, for arbitrary $k_4 \in \big( 0, \max(k_1,k_3)\big)$ there exists $C>0$ such that 
	\begin{nalign}
		I_1 \leqslant  Ce^{-k_4 t} 
		\big( \| \pmb \varphi_0 \|_\Li + \| \psi_0 \|_\Li\big).
	\end{nalign}
	Denoting  $h(s) = \| \pmb \varphi(s)\|_\Li^2 + \|\psi(s)\|_\Li^2$ the integral $I_2$ takes the form
	\begin{nalign}
		I_2 =&C\int_0^t e^{-k_3(t-s)} m_p(t-s)\int_0^s e^{-k_1(s-\tau)} \revone{h(\tau)}\dtau \ds \\ =&\rev{C\int_0^t \revone{h(\tau)}e^{k_1 \tau} e^{-k_3 t}\int_\tau^t e^{(k_3-k_1)s}m_p(t-s)\ds \dtau}.
	\end{nalign} 
	Next, choosing arbitrary $k_4 \in \big( 0, \max(k_1,k_3)\big)$ and following the calculations from~\eqref{I1Es} we obtain the estimates
	\begin{nalign}
		e^{k_1 \tau} e^{-k_3 t}\int_\tau^t e^{(k_3-k_1)s}m_p(t-s)\ds &\leqslant e^{k_1 \tau} e^{-k_3 t} C(1+t-\tau) \left( e^{(k_3 -k_1)t} +  e^{(k_3 -k_1)\tau} \right) \\ 
		&\leqslant	C(1+t-\tau)\left( e^{-k_1(t-\tau)} + e^{-k_3(t-\tau)} \right) \\ & \leqslant Ce^{-k_4(t-\tau)},
	\end{nalign}	
	which imply $I_2 \leqslant 	C\int_0^t 	h(\tau) e^{-k_4(t-\tau)}\dtau$.
\end{proof}
	
\begin{lemma}
\label{thm:NonlinearStabilityPsiOmegaTylda}
	\rev{Assume that $s(\L_p)<0$ and $s\big(d(\cdot)\big)<0$.} Let $(\pmb \varphi, \psi)$ be a solution \eqref{eq:GitSol} to system \eqref{eq:NonlinearStabilityLinearisedProblemWithRest}-\eqref{eq:IniNonLin} satisfying
	\begin{nalign}
		\sup_{t\in [0,\,T]} \|\pmb \varphi(t)\|_\Li + \revone{\sup_{t\in [0,\,T]}}\|\psi(t)\|_\Li \leqslant K \quad \text{for some} \ \ T>0.
	\end{nalign}
	For each $k \in \big(0, \, \min(k_1,k_2, k_3)\big)$ there \revone{exists} $C = C(K,k_4) > 0$ such that
	\begin{nalign}
		\label{eq:NonlinearStabilityPhiPsiLInfinityEstimate}
		\|\pmb \varphi(t) \|_\Li +  \| \psi(t) \|_\Li &\leqslant C e^{-k t}\big( \|\pmb \varphi_0 \|_\Li +  \| \psi_0 \|_\Li \big) \\ &+C \int_0^t e^{-k (t - s)}\left(\|\pmb \varphi(s) \|^2_\Li +  \| \psi(s) \|^2_{L^\infty({\Omega})}\right)\ds
	\end{nalign}
	for all $t\in [0, \, T]$.	
\end{lemma}

\begin{proof} 
	Applying Lemma \ref{thm:NonlinearStabilityOneToSquareEstimate} with $p>N/2$ and with arbitrary $k_4 \in \big(0, \, \min(k_1,k_3)\big)$ to the linear term containing $\| \pmb \varphi\|_\Lp{p}$ on the right-hand side of the inequality in Lemma~\ref{thm:NonlinearStabilityPsiByLinAndSquareEstimate} \rev{yields}
	\begin{nalign}
		\label{PsiIneqLi}
		\| \psi(t) \|_\Li  &\leqslant C e^{-k_4t} \big(\|\psi_0\|_\Li + \|\pmb \varphi_0 \|_\Li\big) \\ &+ C \int_0^t e^{-k_4(t-s)} \big( \| \pmb \varphi (s)\|^2_\Li + \|  \psi (s)\|^2_\Li \big) \ds. 
	\end{nalign} 
	Next, we compute the $L^\infty$-norm of the inequality provided by Lemma \ref{thm:NonlinearStabilityFirstPointwiseEstimate} for a constant $k_5\in\big(0, \, \min(k_2,k_4)\big)$ and obtain
	\begin{nalign}
		\label{eq:NonlinearStabilityPhiLInfinityK5Estimate}
				\|\pmb \varphi(t)\|_\Li &\leqslant C e^{-k_5 t} \|\pmb \varphi_0\|_\Li \\&+ C \int_0^t e^{-k_5(t-s)} \big(\| \psi(s)\|_\Li + \|\psi(s)\|_\Li^2 + \|\pmb \varphi (s)\|_\Li^2 \big)\ds.
	\end{nalign}
	Now, it suffices to  estimate the linear term in inequality~\eqref{eq:NonlinearStabilityPhiLInfinityK5Estimate} using inequality \eqref{PsiIneqLi}. Therefore after changing \rev{the order of integration} we obtain
	\begin{nalign}
		\int_0^t e^{-k_5(t-s)} & \| \psi(s)\|_\Li \ds 	\leqslant Cte^{-k_5t} \revone{\big(\|\pmb \varphi_0 \|_\Li +\| \psi_0\|_\Li)} \\ &+\rev{C \int_0^t \revone{e^{-k_5(t-s)}} \int_0^s \revone{e^{-k_4(s-\tau)}}\big(\|\pmb \varphi(\tau) \|^2_{L^\infty(\Omega)} +  \| \psi(\tau) \|^2_\Li\big)\dtau \ds} \\
		&\leqslant Cte^{-k_5t} \| \psi_0\|_\Li + C \int_0^t e^{-k_5(t-\tau)}\big(\|\pmb \varphi(\tau) \|^2_\Li +  \| \psi(\tau) \|^2_\Li\big)\dtau.
	\end{nalign}
	Thus, for arbitrary number $k \in \big(0, \, \min(k_4, k_5)\big)$, there exists a constant $C>0$ such that
	\begin{nalign}
	\label{eq:NonlinearStabilityK5PsiEstimate}
		\int_0^t e^{-k_5(t-s)} \| \psi(s)\|_\Li \ds &\leqslant Ce^{-kt} \revone{\big(\|\pmb \varphi_0 \|_\Li +\| \psi_0\|_\Li)}  \\ &+C \int_0^t e^{-k(t-\tau)}\big(\|\pmb \varphi(\tau) \|^2_\Li +  \| \psi(\tau) \|^2_\Li\big)\dtau.
	\end{nalign} 
	By applying inequality \eqref{eq:NonlinearStabilityK5PsiEstimate} to inequalities \eqref{eq:NonlinearStabilityPhiLInfinityK5Estimate} and \eqref{PsiIneqLi} we complete the proof of estimate \eqref{eq:NonlinearStabilityPhiPsiLInfinityEstimate}.
\end{proof}

Next, let us recall the well-known fact from stability theory which we prove for completeness of the exposition.
\begin{lemma}
	\label{thm:NonlinearStabilityClassicalEstimate}
	Let $h = h(t)$ be a nonnegative continuous function on $\R_+$ satisfying
	\begin{nalign}
		\label{eq:HIneq}
		h(t) \leqslant  Ce^{-kt} h(0) +  C \int_0^t e^{-k(t-s)}h^2(s)\ds \quad \text{for each} \quad t>0
	\end{nalign}	
	and some constant $C\geqslant 1$. If $h(0)\leqslant k/(4C^2)$, then $h(t)\leqslant  Ch(0) e^{-\frac{k}{2}t}$ for all $t>0$.  
\end{lemma}

\begin{proof}
	Let $T>0$ be such that $h(t) < k/(2C)$ for every $t\in [0,\, T]$. After multiplying both sides of inequality \eqref{eq:HIneq} by $e^{kt}$ we obtain
	\begin{nalign}
		h(t)e^{kt} \leqslant  C h(0) +  \frac{k}{2} \int_0^t e^{ks}h(s)\ds \quad \text{for each} \quad t\in[0,\, T]
	\end{nalign}  
	and using the Gr\"onwall inequality we conclude that $h(t) \leqslant  C h(0) e^{-\frac{k}{2}t}$ for all $t\in[0,\, T]$. 
	
	In order to show that this inequality holds true for all $t>0$ we show that $h(t) < k/(4C)$. Indeed, \rev{if there \revone{exists} }$t_0\in[0, \, T]$ such that $h(t_0) = k/(4C)$ then
	\begin{nalign}
		  \frac{k}{4C} = h(t_0)  \leqslant C \frac{k}{4C^2} e^{-\frac{k}{2}t_0}< \frac{k}{4C},
	\end{nalign}
	which is a contradiction.
\end{proof}

\begin{proof}[Proof of Theorem \ref{thm:NonlinearStability}]
	Let $(\pmb \varphi, \psi)$ be a solution \eqref{eq:HIneq} to problem \eqref{eq:NonlinearStabilityLinearisedProblemWithRest} with an initial condition $(\pmb \varphi_0, \psi_0) \in \Li^n \times C(\oo)$.
	Fix a number $K>0$ from Assumption \ref{ass:SquareEs} and assume that the initial condition satisfies 
	\begin{nalign}
		\|\pmb \varphi_0\|_\Li + \|\psi_0\|_\Li <\frac{K}{4C^2}.
	\end{nalign} 
	By continuity in \eqref{eq:GitSol}, we may choose maximal $T>0$ such that
	\begin{nalign}
		\sup_{t\in [0,\, T]}\|\pmb \varphi(t)\|_\Li +\sup_{t\in [0,\, T]} \|\psi(t)\|_\Li \leqslant K.
	\end{nalign}
	If $T<\infty$  then $\|\pmb \varphi(T)\|_\Li + \|\psi(T)\|_\Li = K$. On the other hand by Lemma \ref{thm:NonlinearStabilityPsiOmegaTylda} and Lemma \ref{thm:NonlinearStabilityClassicalEstimate}, we obtain
	\begin{nalign}
		\label{eq:PhiPsiDecay}
		\|\pmb \varphi(t)\|_\Li + \|\psi(t)\|_\Li \leqslant C \big(\|\pmb \varphi_0\|_\Li + \|\psi_0\|_\Li\big) e^{-\frac{k}{2} t} \leqslant \frac{K}{4C}
	\end{nalign}
	for $t\in [0, \, T]$ which is a contradiction for $t=T$ because $C\geqslant1$. Hence the inequality~\eqref{eq:PhiPsiDecay} holds true for all $t>0$.  
\end{proof}

\section{Stability of discontinuous stationary solutions} 
\label{sec:Application}

\subsection{Existence of local-in-time solutions} 

In this work, we consider mild solutions of problem \eqref{eq1}-\eqref{nonlinearity}, that is, a solution to the corresponding system of integral equations
\begin{nalign}
	\label{eq:InitialIntegralSystem}
	\pmb u &= \pmb u_0 + \int_0^t \pmb f\big(\pmb u(s), v(s)\big) \ds, \\  
	v & = \rev{e^{\gamma\Delta_\nu t} v_0 + \int_0^t  e^{\gamma\Delta_\nu (t-s)} g\big(\pmb u(s), v(s)\big) \ds.}
\end{nalign} 

\begin{theorem}
	\label{thm:ExPar}
	For every $\pmb u_0 \in \Li^n$ and $v_0 \in C(\oo)$ there exists $T>0$ such that system 
	\eqref{eq:InitialIntegralSystem} has a unique solution $\pmb u \in C\big([0,\,T],\, \Li^n \big)$ and $v\in C\big([0,\,T], \, C(\oo)\big)$.
\end{theorem}

\begin{proof}
	First, we recall that the operator $\big(\Delta_\nu, D(\Delta_\nu)\big)$ with the domain
	\begin{nalign}
		D(\Delta_\nu) = \left\lbrace  u\in \bigcap_{p\geqslant 1} W^{2,p}_{loc}(\Omega): \quad u, \, \Delta u \in C(\overline{\Omega}), \quad  \partial_{\nu} u = 0 \text{ on } \partial\Omega  \right\rbrace,
	\end{nalign}
	is a sectorial operator on the space $C(\overline{\Omega})$, see {\cite[Corollary 3.1.24]{MR3012216}}. Now, it is sufficient to apply the Banach fixed point theorem to system \eqref{eq:InitialIntegralSystem} with the locally Lipschitz nonlinearities $\pmb f$ and $g$. In this approach the \rev{Bochner} integrals on the right hand side of equation \eqref{eq:InitialIntegralSystem} are convergent in the corresponding spaces. Indeed, for every $\pmb u \in  C\big([0,\,T],\, \Li^n \big)$ and $v\in C\big([0,\,T], \, C(\oo)\big)$  we have $\pmb f\big( \pmb u, \, v) \in C\big( [0,T], \, \Li^n\big)$ and consequently the integral is convergent 
	\begin{nalign}
		\int_0^t \pmb f\big(\pmb u(s), v(s)\big) \ds\in C\big( [0,\,T], \, \Li^n\big).
	\end{nalign} 
	Analogously, we have $g\big( \pmb u, \, v) \in C\big( [0,T], \, \Li\big)$ and hence by the well-known property of the Neumann semigroup we obtain 
	\begin{nalign}
		\rev{w(s) \equiv e^{\gamma\Delta_\nu (t-s)} g\big(\pmb u(s), v(s)\big) \in C(\oo)}
	\end{nalign} 
	for each $t\in (0,\, T)$ and $s\in[0,\, t)$. Therefore  
	\begin{nalign}
		w \in C\big( [0,\, t); \, C(\oo) \big) \cap L^\infty\big([0, \, t], \, \Li\big)
	\end{nalign}
	and for every $\varepsilon >0$
	\begin{nalign}
		\int_0^{t-\varepsilon}  e^{\gamma\Delta_\nu (t-s)} g\big(\pmb u(s), v\big) \ds \in C\big( [0, T -\varepsilon]; \, C(\oo) \big).
	\end{nalign}
	Since $g(\pmb u, \, v) \in {L^\infty([0,T]; \Li)} $ we have
	\begin{nalign}
		\rev{\int_0^{t-\varepsilon}  e^{\gamma\Delta_\nu (t-s)} g\big(\pmb u(s), v\big) \ds \xrightarrow{\varepsilon \to 0} \int_0^{t}  e^{\gamma\Delta_\nu (t-s)} g\big(\pmb u(s), v(s)\big) \ds}
	\end{nalign}
	locally uniformly on $[0, \,T) \times \oo$.
\end{proof}

\subsection{Proofs of stability theorems} 

We are in \rev{a position} to apply  theorems from Sections \ref{sec:LinearEquation} and \ref{sec:Nonlinear Stability} to obtain a stability  of stationary solutions to the initial-boundary value problem for general reaction-diffusion-ODE system \eqref{eq1}-\eqref{nonlinearity}.

\begin{proof}[Proof of Theorem \ref{thm:ApplicationSystemStabilityDiscontinuous2}]
	Let $(\overline{ \pmb U}, \overline{V})\in \R^{n+1}$ be a constant solution to problem \eqref{DisProbDef} with the corresponding  linearisation matrices \eqref{ass:Matrix} and satisfying Assumptions \ref{ass:StatSol}, \ref{ass:DiscontinuousStationaryTwoBranches}, \ref{ass:LinearStability} and \ref{ass:LinearStabilityGamma}. Denote by $\rho>0$ a constant required in Theorem~\ref{thm:ResEsNonCon}. 
	
	By the assumptions $s\big(\pmb f_{\pmb u} \left({ \pmb k_1(\overline{V})},  \overline{ V}\right)\big) < 0$ and $s\big(\pmb f_{\pmb u} \left({ \pmb k_2(\overline{V})},  \overline{ V}\right)\big) < 0$, we may choose $\varepsilon_1>0$ such that 
	\begin{nalign} 
		\label{P2VarCon}
		\sup_{w\in{B_{\varepsilon_1}(\overline{ V})}} s\big(\pmb f_{\pmb u} ({ \pmb k_1(w)}, \, w)\big)  <  0 \quad \text{and} \quad
		\sup_{w\in{B_{\varepsilon_1}(\overline{ V})}} s\big(\pmb f_{\pmb u} ({ \pmb k_2(w)}, \, w)\big) < 0.
	\end{nalign}
	Similarly, by continuity we have
	\begin{nalign}
		\label{P2VarNeigh}
		|\Omega|^\frac{1}{N}\sup_{w\in{B_{\varepsilon_1}(\overline{ V})}} \left|
		\begin{pmatrix}
			\pmb f_{\pmb u} \left({ \pmb k_1(w)}, \, w\right)
			&\pmb f_{v} \left({ \pmb k_1(w)}, \, w\right) \\ 
			g_{\pmb u} \left({ \pmb k_1(w)}, \, w\right)
			&g_{v} \left({ \pmb k_1(w)}, \, w\right)  
		\end{pmatrix}
		-
		\begin{pmatrix}
			\pmb A_0 
			&\pmb B_0 \\ 
			\pmb C_0 
			&d_0  
		\end{pmatrix}\right|
		< \frac{\rho}{8}.
	\end{nalign} 
	Next, we choose $\delta_1>0$ so small to have
	\begin{nalign}
		\label{P2VarNeigh2}
		\delta_1^\frac{1}{N}\sup_{w\in{B_{\varepsilon_1}(\overline{ V})}} \left|
		\begin{pmatrix}
			\pmb f_{\pmb u} \left({ \pmb k_2(w)}, \, w\right)
			&\pmb f_{v} \left({ \pmb k_2(w)}, \, w\right) \\ 
			g_{\pmb u} \left({ \pmb k_2(w)}, \, w\right)
			&g_{v} \left({ \pmb k_2(w)}, \, w\right)  
		\end{pmatrix}
		-
		\begin{pmatrix}
			\pmb A_0 
			&\pmb B_0 \\ 
			\pmb C_0 
			&d_0  
		\end{pmatrix}\right|
		< \frac{\rho}{8}.
	\end{nalign} 

	Now, we consider a discontinuous solution $\UV$ to problem 		 \eqref{DisProbDef} constructed in Theorem \ref{DisExBan} with $\varepsilon < \min(\varepsilon_0, \varepsilon_1)$ and $|\Omega_2| < \min(\delta, \delta_1)$. Let us check that the assumptions of Theorem~\ref{thm:ResEsNonCon} are satisfied in the case of matrices from Remark \ref{Rem:MatLin}. Indeed, Assumptions \ref{ass:LinearStability}, \ref{ass:LinearStabilityGamma} and inequality \eqref{P2VarCon} imply inequalities \eqref{eq:ThmNeq} and \eqref{eq:Acon}. Moreover, inequalities \eqref{P2VarNeigh} and \eqref{P2VarNeigh2} provide
	\begin{nalign}
		\|\pmb A - \pmb A_0\|_N  &\leqslant  \|\pmb A - \pmb A_0\|_{L^{N}(\Omega_1)} + \|\pmb A - \pmb A_0\|_{L^{N}(\Omega_2)} \\ &\leqslant |\Omega_1|^\frac{1}{N} \|\pmb f_{\pmb u} (\pmb k_1(V), \, V) - \pmb A_0\|_{L^\infty(\Omega_1)} +|\Omega_2|^\frac{1}{N}\|\pmb f_{\pmb u} (\pmb k_2(V), \, V) - \pmb A_0\|_{L^\infty(\Omega_2)} \\  &\leqslant \rho/4.
	\end{nalign}
	We show analogously the following inequalities 
	\begin{nalign}
		\|\pmb B - \pmb B_0\|_\Lp{N} \leqslant\rho/4,\quad  \|\pmb C - \pmb C_0\|_\Lp{N} \leqslant\rho/4\quad \text{and} \quad\| d - d_0\|_\Lp{N} \leqslant\rho/4.
	\end{nalign}
	By Theorem \ref{thm:ResEsNonCon}, the stationary solution $\UV$ is exponentially asymptotically stable in $\Lp{}^{n+1}$ for each $p\in (1,\infty)$.
\end{proof}

\begin{proof}[Proof of Theorem \ref{thm:ApplicationSystemStabilityDiscontinuous}]
	It suffices to check assumptions of Theorem~\ref{thm:NonlinearStability}  in the case of problem \eqref{eq:ReactionDiffusionLinearisedProblemWithRest}-\eqref{eq:LinIni}. By Remark \ref{rem:Conti}, a stationary solution constructed in Theorem~\ref{DisExBan} satisfies $(\pmb U, V)\in \Li^n\times C(\oo)$. Thus, problem \eqref{eq:ReactionDiffusionLinearisedProblemWithRest}-\eqref{eq:LinIni} has a solution 
	\begin{nalign}
		\pmb \varphi = \pmb u - \pmb U\in C\big([0,\,T],\, \Li^n \big) \quad \text{and}\quad \psi = v - V\in C\big([0,\,T], \, C(\oo)\big),
	\end{nalign} 
	for some $T>0$ by applying Theorem \ref{thm:ExPar}. Due to the properties of the Taylor remainders, Assumption \ref{ass:SquareEs} holds true for each $K>0$. By Theorem \ref{thm:ApplicationSystemStabilityDiscontinuous2}, the corresponding linearisation operator $\big(\L_p,D(\L_p)\big)$ satisfies $s(\L_p)<0$ for each $p\in(1, \infty)$. It remains to prove that $s\big(g_v(\pmb U(\cdot), \, V(\cdot)\big) <0$ provided $\varepsilon>0$ is sufficiently small. This however is an immediate consequence of condition \eqref{eq:Dcon}, inequalities \eqref{DisVar} and the $C^2$-regularity of nonlinear terms. By Theorem~\ref{thm:NonlinearStability} applied to $\pmb \varphi = \pmb u - \pmb U$ and $\psi = v - V$, if 
	\begin{nalign}
		\|\pmb u_0 - \pmb U\|_\Li + \|v_0-V\|_\Li <K/(4C^2),
	\end{nalign}
	then
	\begin{nalign}
		\|\pmb u - \pmb U\|_\Li + \|v - V\|_\Li \leqslant C \big(\|\pmb u_0 - \pmb U\|_\Li + \|v_0 - V\|_\Li\big) e^{-\frac{k}{2} t},
	\end{nalign}
	for all $t\in [0,\, T]$. Since the constants here are independent of $T$ this solution and the inequality can be extended to all $t>0$.
\end{proof}

\section{Example} 

\label{sec:Examples}

\subsection{Model example}

We apply results from this work to the following reaction-diffusion-ODE model
\begin{nalign}
	\label{eq:ExamplesFitzHugh}
	u_t &= u(1-u)(u-\beta) - v, && \text{for} \; x \in \overline{\Omega}, \; t>0, \\
	v_t &= \gamma  \Delta_\nu v + \sigma u - \delta v - \rho, && \text{for} \; x \in \Omega, \; t>0,
\end{nalign}
supplemented with the initial conditions
\begin{nalign}
	\label{eq:ExamplesIni}
	u(0,x) = u_0(x), \quad 
	v(0,x) = v_0(x), 
\end{nalign}   
where $\beta$, $\sigma$, $\delta$, $\rho$ are positive constants and $\Omega \subseteq \R^{N}$ (with arbitrary $N\geqslant 1$) is an open bounded domain. Here we are motivated by the FitzHugh-Nagumo system
\begin{nalign}
		u_t = \varepsilon^{2} u_{x x}+u(1-u)(u-\beta)-v, \quad 
		v_t=\sigma u-\delta v-\rho
\end{nalign}
which was used to model pulse propagation in excitable media, see \textit{e.g.}~\cite{FITZHUGH1961445, MR2054989, 4066548, MR1779040}. The following generalisation of the FitzHugh-Nagumo system 
\begin{nalign}
		\label{eq:FitzOrg}
		u_t = \varepsilon^{2} \Delta_\nu u+u(1-u)(u-\beta)-v, \quad 
		v_t = \gamma \Delta_\nu v+\sigma u-\delta v-\rho
\end{nalign}
was studied in the context of pattern formation (\cite{ MR2207551, MR1153026, MR1954510}). This system has traveling wave
solutions and other types of solutions with  a complex behavior (see, e.g., \cite{MR2110206, MR2765427, MR3527621} and
references therein). In particular, the work \cite{MR1954510} contains some results on a behaviour of solutions to system \eqref{eq:FitzOrg} when $\varepsilon \to 0$.

A study of the one dimensional version of the reaction-diffusion-ODE system \eqref{eq:ExamplesFitzHugh} can be also motivated by 
series of papers  \cite{MR3583499, MR3679890,MR3220545, Kthe2020HysteresisdrivenPF, MR2205561, MR3329327, MR3059757} where  similar  receptor-based models were considered. 

The paper \cite{MR3973251} contains a construction of a family of discontinuous steady states of system \eqref{eq:ExamplesFitzHugh} in the one dimensional case and the proof of their stability. Below, we explain how using theory developed in this work we obtain analogous results in higher dimensions.
	
\subsection{Constant solutions}

An analysis of stationary solutions to system \eqref{eq:ExamplesFitzHugh} begins by the assumption that it has a constant stationary solution, namely a vector $(\overline{U}, \overline{V}) \in \R^2$ satisfying 
\begin{nalign}
	\overline{ U}(1-\overline{ U})(\overline{U}-\beta) - \overline{  V}  = 0\qquad  \text{and} \qquad  \sigma \overline{U} - \delta \overline{  V} - \rho   = 0.
\end{nalign}
On Fig. \ref{fig:Histeresis}, we choose the parameters $\beta, \sigma, \delta, \rho$ in such a way to have three constant solutions $\left(\overline{U}_1,\overline{V}_1\right)$, $\left(\overline{U}_2, \overline{V}_2\right)$ and $\left(\overline{U}_3, \overline{V}_3\right)$. 

\begin{figure}[h]
	\includegraphics[width=0.7\textwidth]{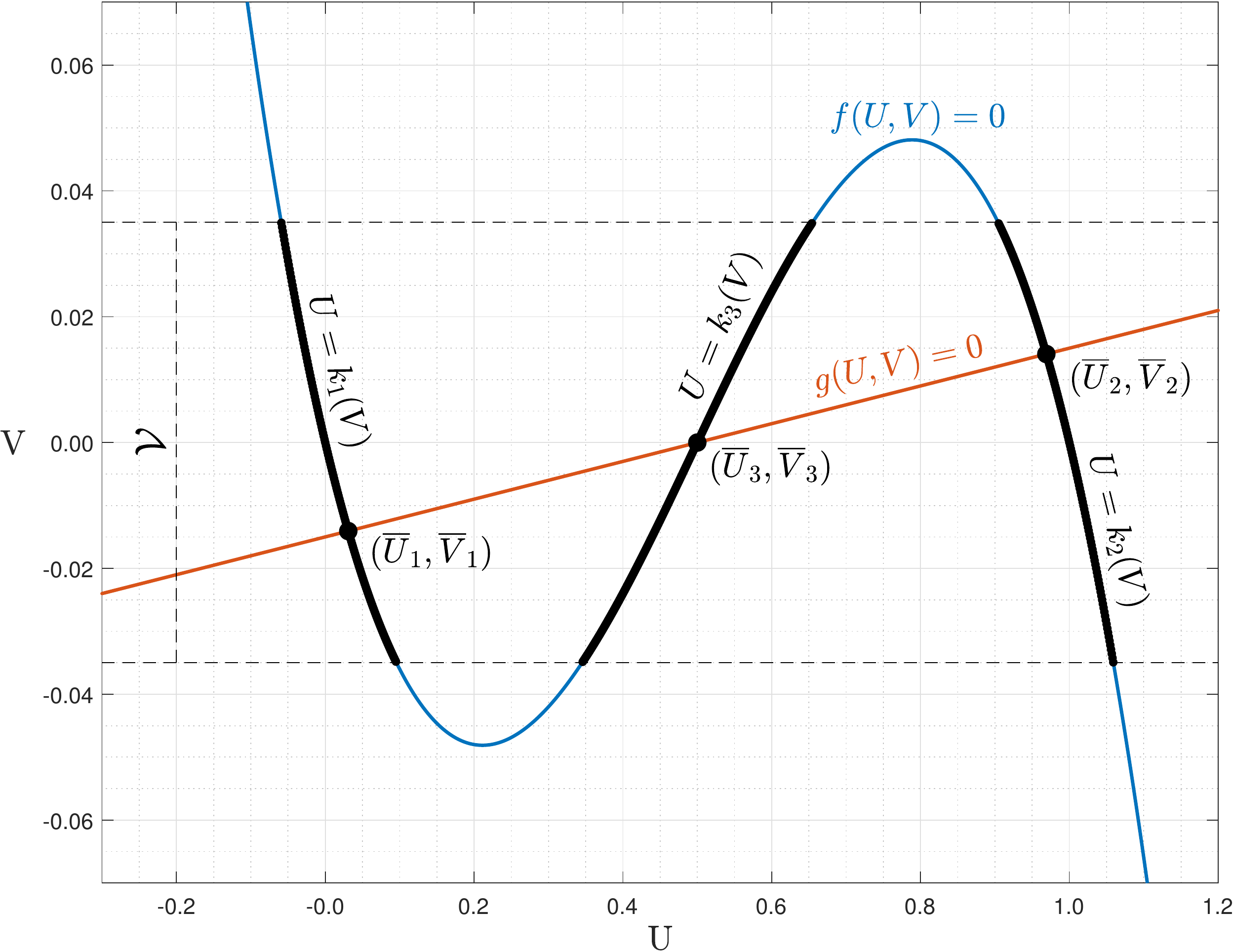}
	\caption{The nullclines for system \eqref{eq:ExamplesFitzHugh}, namely, the curves  \newline  $f(U,V) \equiv { U}(1-{ U})({U}-\beta) - {  V} = 0$ and $g(U,V)\equiv \sigma U-\delta V-\rho = 0$.}
	\label{fig:Histeresis}
\end{figure}

\subsection{Regular stationary solutions}
By our previous work \cite{CMCKS01}, a stationary solution $(U,V)$ is called regular if there exist a $C^2$-function $k$ such that $U(x)= k(V(x))$. Such regular stationary solutions to problem \eqref{eq:ExamplesFitzHugh}
can be immediately  obtained from  \cite[Prop.~2.5]{CMCKS01} by a bifurcation around the constant solution $(\overline{U},\overline{V})$ provided
\begin{nalign}
	\frac{1}{\det f_u(\overline{U}, \overline{ V})}\det \begin{pmatrix} f_u(\overline{U}, \overline{ V}) & f_v(\overline{U}, \overline{ V}) \\ g_u(\overline{U}, \overline{ V}) & g_v(\overline{U}, \overline{ V})  \end{pmatrix} = \gamma \mu_k,
\end{nalign}
where $f(u,v) = u(1-u)(u-\beta) - v $, $g(u,v) = \sigma u - \delta v - \rho$ and $\mu_k$ is an eigenvalue of $-\Delta_\nu$, and with the diffusion coefficient $\gamma>0$.  
It follows from our work \cite{CMCKS01} that all regular stationary solutions to system \eqref{eq:ExamplesFitzHugh} are unstable.

\subsection{Discontinuous stationary solutions}

Now, we fix a constant stationary solutions, say, $(\overline{U}, \overline{ V}) = (\overline{U}_1, \overline{V}_1)$ on Fig.~\ref{fig:Histeresis}. Following Assumption~\ref{ass:DiscontinuousStationaryTwoBranches}, we consider two different branches of solutions to equation ${ U}(1-{ U})({U}-\beta) - {  V} = 0$ in a neighbourhood of $\overline{V}_1$. Namely, we choose a set $\V\subseteq \R$ and functions $k_1, k_2\in C^2(\V,\R)$ such that $\overline{ V}_1 \in \V$ and
\begin{nalign}
	\begin{matrix}
	k_1(w)\big(1-{ k_1(w)}\big)\big({k_1(w)}-\beta\big) - {  V} = 0, \\ 	k_2(w)\big(1-{ k_2(w)}\big)\big({k_2(w)}-\beta\big) - {  V} = 0,
	\end{matrix} \quad \text{for each} \quad w\in\V
\end{nalign}
see bold curves on Fig. \ref{fig:Histeresis}. By Theorem \ref{DisExBan}, for an arbitrary decomposition \rev{$\Omega_1\subset \Omega$ and $\Omega_2=\Omega\setminus \overline\Omega_1$}
with sufficiently small $|\Omega_2|>0$ and for all small $\varepsilon>0$ there exist a stationary solution to system \eqref{eq:ExamplesFitzHugh} with the following properties
	\begin{itemize}
	\item $(U, V)\in L^\infty(\Omega)^n \times \W2p$;
	\item $V=V(x)$ is a weak solution to problem
	\begin{nalign}
		\gamma\Delta_\nu V +  \sigma U - \delta V -\rho = 0 \quad \text{for} \quad x\in \Omega
	\end{nalign}
	and  		
	\begin{nalign}
		U (x) = \begin{cases}
			k_1\big(V(x)\big), \quad x\in \Omega_1, \\
			k_2\big(V(x)\big), \quad x\in \Omega_2,
		\end{cases}
	\end{nalign} 
	satisfies ${ U}(1-{ U})({U}-\beta) - {  V} = 0$ for almost all $x\in \Omega$;
	\item the couple $(U,V)$ satisfies
	\begin{nalign}
		\|V - \overline{V}_1 \|_\Li < \varepsilon \quad \text{and} \quad \| U -   \overline{U}_1 \|_{L^\infty(\Omega_1)} + \| U - k_2 (\overline{V}_1) \|_{L^\infty(\Omega_2)}  < C\varepsilon.
	\end{nalign}
\end{itemize}

 An analogous analysis can be repeated by choosing the constant solution $ (\overline{ U}_1, \overline{ V}_1)$ together with the branches $k_1$ and $k_3$, in the neighbourhood $\V \subseteq \R$, see Fig.~\ref{fig:Histeresis}. A~discontinuous stationary solution can be also obtained, if we use a constant solution $(\overline{ U}_2, \overline{ V}_2)$ or $(\overline{ U}_3, \overline{ V}_3)$ with two branches of either $k_1$ or $k_2$ or $k_3$ with a suitable neighbourhood $\V \subseteq \R$ of the points $\overline{ V}_2$ or $\overline{ V}_3$. Here, we can also apply Remark \ref{thm:3Branches} and consider three branches $ k_1,  k_2, k_3$ to obtain discontinuous stationary solutions on three different sets $\Omega_1, \Omega_2, \Omega_3$.

\subsection{Stability of stationary solutions}
Let us discuss sufficient conditions for stability of the discontinuous stationary solution $(U,V) = \big(U(x),V(x)\big)$ constructed around the point $(\overline{ U}_1, \overline{ V}_1)$ with branches $k_1$ and $k_2$. Assumptions~\ref{ass:LinearStability} and condition~\eqref{eq:SecBranch} reduce to the inequalities
\begin{nalign}
	\label{eq:ExampleStability}
	- 3\overline{U}_1^2 + 2 \overline{U}_1(1+\beta&) - \beta <0, \quad - 3k_2(\overline{V}_1)^2 + 2 k_2(\overline{V}_1)(1+\beta) - \beta <0, \\
	&s\begin{pmatrix} 	- 3\overline{U}_1^2 + 2 \overline{U}_1(1+\beta) - \beta & -1 \\ \sigma & -\delta \end{pmatrix} < 0.
\end{nalign}
The first two hold true because the function $U(1-U)(U-\beta)$ is strictly decreasing around points $\overline{ U}_1$ and $k_2(\overline{V}_1)$, see Fig. \ref{fig:Histeresis}. The matrix in the third inequality has eigenvalues with negative real parts if its trace is negative and the determinant is positive which is equivalent to the inequalities 
\begin{nalign}
		- 3\overline{U}_1^2 + 2 \overline{U}_1(1+\beta) - \beta - \delta <0 \quad \text{and} \quad -\delta(- 3\overline{U}_1^2 + 2 \overline{U}_1(1+\beta) - \beta)+\sigma > 0. 
\end{nalign}
Both of them are immediately obtained from first inequality in \eqref{eq:ExampleStability}. Choosing $\gamma>0$ sufficiently large as in Assumption \ref{ass:LinearStabilityGamma} (this requirement can be relaxed by applying Proposition \ref{thm:LinearStabilitystantNoncoefficientsN1})  as well as $\varepsilon>0$ and $|\Omega_2|>0$ sufficiently small we obtain by Theorem~\ref{thm:ApplicationSystemStabilityDiscontinuous2} that the solution $(U,V)$ is linearly exponentially stable. Moreover, since the conditions \eqref{eq:Dcon} reduce to the inequality $-\delta <0$, this solution is also nonlinearly stable in $\Li^{n+1}$ by Theorem~\ref{thm:ApplicationSystemStabilityDiscontinuous}.

On the other hand, every discontinuous stationary solution containing the branch $k_3$ is linearly unstable. This is a consequence of the fact that $3k_3(w)^2 + 2 k_3(w)(1+\beta) - \beta >0$ for each $w\in \V$ and $3k_3(U)^2 + 2 k_3(U)(1+\beta) - \beta$ belongs to the spectrum of the corresponding linearisation operator by equation~\eqref{thm:SpectrumOfOperatorL}, see Remark \ref{thm:AutoPoin} for more comments.   

\subsection{Numerical illustrations} 
We conclude this work by numerical simulations of solutions to problem \eqref{eq:ExamplesFitzHugh}-\eqref{eq:ExamplesIni} with $\beta, \delta, \sigma, \rho$ as on Fig.~\ref{fig:Histeresis} and with large $\gamma >0$. To obtain Fig.~\ref{fig:UVpi} and Fig.~\ref{fig:UVRandom} we used the explicit finite difference Euler method. In both cases, we choose an initial conditions $(u_0, v_0)$ as a small perturbations of the constant solution $(\overline{ U}_3, \overline{ V}_3)$. We decompose the domain $\Omega=[0,1]^2 = {\Omega_1 \cup \Omega_2}$ with $|\Omega_2|>0$ sufficiently small and set  
\begin{nalign}
	\label{ExampleIni}
	u_0 <\overline{ U}_3 \text{ and } v_0<\overline{ V}_3 \text{ on } \Omega_1 \quad \text{as well as} \quad u_0 >\overline{ U}_3 \text{ and } v_0>\overline{ V}_3  \text{ on } \Omega_2.
\end{nalign} 
We choose the $\pi$-shape set $\Omega_2$ on Fig.~\ref{fig:UVpi} and a random small set $\Omega_2$ on Fig.~\ref{fig:UVRandom}. The graphs of the corresponding solutions on both Figures \ref{fig:UVpi} and \ref{fig:UVRandom} were obtained for large values of $t>0$ and, by stability results from this work, they are good approximations of discontinuous stationary solutions. Here we observe a graphical illustration of inequalities \eqref{DisVar}, namely, $U$ stays close to $\overline{ U}_1 = k_1(\overline{ V}_1)$ on the set $\Omega_1$ and close to $k_2(\overline{ V}_1)$ on the set $\Omega_2$. \rev{By colours, we denote different values of $u$ and $v$: blue and black denote that $u$ and $v$ stay close to points $(\overline{U}_1, \overline{V}_1)$  and $(k_2(\overline{  V}_1),\overline{  V}_1)$, respectively; by yellow, green and grey, we mark small perturbations of those values.}

\begin{figure}[!h]
	\includegraphics[width=1\textwidth]{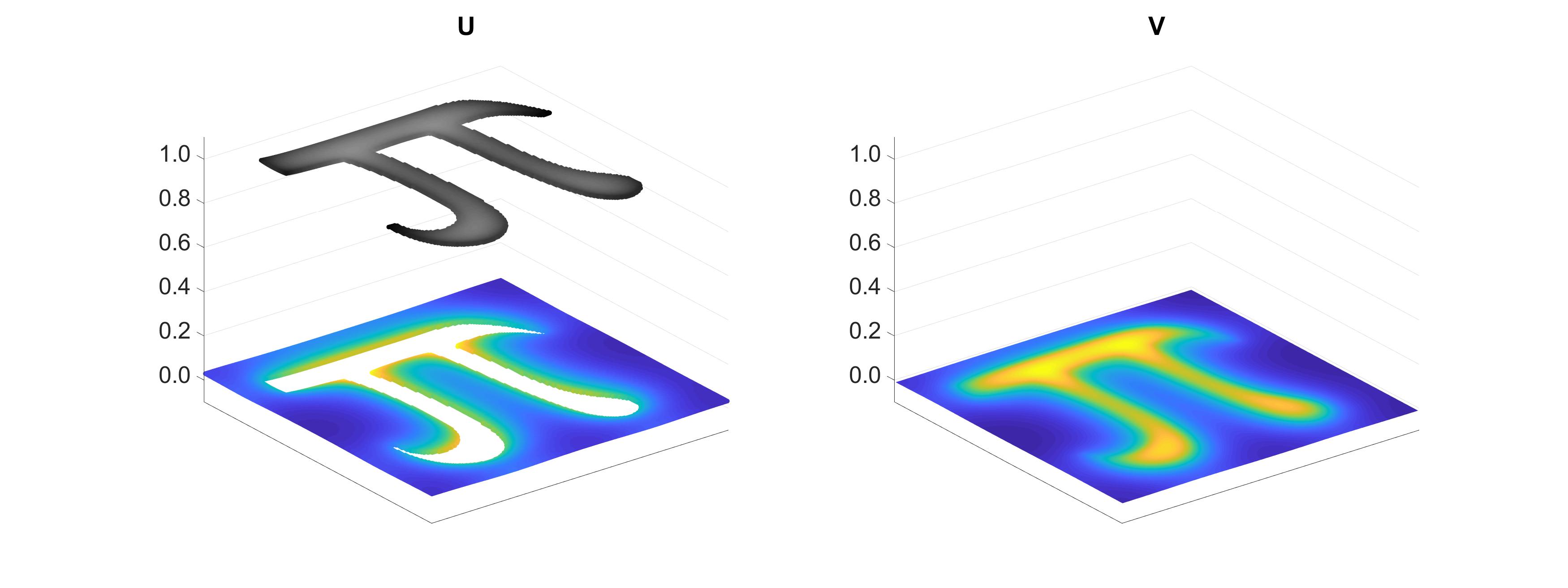}
	\caption{The solution $u,v$ to problem \eqref{eq:ExamplesFitzHugh}-\eqref{eq:ExamplesIni} for sufficiently large $t>0$ and with the initial condition satisfying inequalities \eqref{ExampleIni}. Here, we choose the particular $\pi$-shape set $\Omega_2$ to emphasize that every shape is allowed.}
	\label{fig:UVpi}
\end{figure}

\begin{figure}[!h]
	\includegraphics[width=1\textwidth]{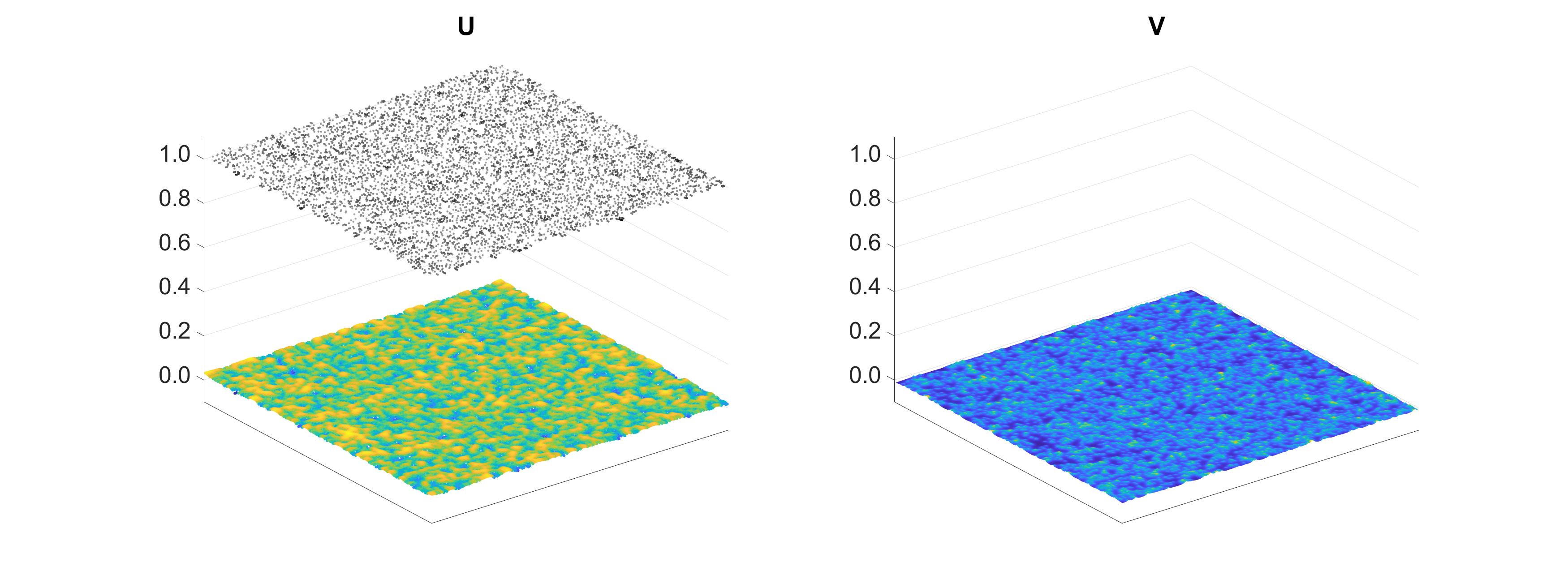}
	\caption{The solution $u,v$ to problem \eqref{eq:ExamplesFitzHugh}-\eqref{eq:ExamplesIni} for sufficiently large $t>0$ and for a random initial condition.}
	\label{fig:UVRandom}
\end{figure}

\appendix 




\end{document}